\documentclass{emsprocart}
\usepackage{wasysym}

\contact[brundan@uoregon.edu]{Jonathan Brundan, Department of
  Mathematics, University of Oregon, Eugene, OR 97403, USA}

\numberwithin{equation}{section}

\newtheorem{Theorem}{Theorem}[section]
\newtheorem{Corollary}[Theorem]{Corollary}
\newtheorem{Lemma}[Theorem]{Lemma}

\theoremstyle{definition}

\renewcommand{\t}{\mathtt{T}}
\newcommand{\Pol}{P\!ol}
\newcommand{\Sym}{S\hspace{-0.2mm}ym}
\newcommand{\K}{\mathbb{K}}
\newcommand{\Z}{\mathbb{Z}}
\newcommand{\N}{\mathbb{N}}
\newcommand{\Q}{\mathbb{Q}}
\newcommand{\C}{\mathbb{C}}
\newcommand{\Hom}{\operatorname{Hom}}
\newcommand{\End}{\operatorname{End}}

\newcommand{\Dim}{\operatorname{Dim}}
\newcommand{\KP}{\operatorname{KP}}
\newcommand{\height}{\operatorname{ht}}
\newcommand\NH{N\!H}
\newcommand{\W}{\langle I \rangle}
\newcommand{\Rep}{\operatorname{Rep}}
\newcommand{\Proj}{\operatorname{Proj}}
\newcommand{\Ch}{\operatorname{Ch}}
\newcommand{\ch}{\operatorname{ch}}
\newcommand{\Ind}{\operatorname{Ind}}
\newcommand{\Res}{\operatorname{Res}}
\def\homog{\Omega}

\def\f{\mathbf{f}}
\def\bi{\text{\boldmath$i$}}
\def\bj{\text{\boldmath$j$}}
\def\bk{\text{\boldmath$k$}}

\def\b{\mathtt{b}}

\title[Quiver Hecke Algebras]{Quiver Hecke Algebras and Categorification}

\author[Jonathan Brundan]{Jonathan Brundan\thanks{Supported in part by
    NSF grant no. DMS-1161094.}}

\begin{document}

\begin{abstract}
This is a brief introduction to the 
quiver Hecke algebras of Khovanov, Lauda and Rouquier, emphasizing their application to the
categorification of quantum groups.
The text is based on lectures given by the author at the 
ICRA workshop in Bielefeld in August, 2012.
\end{abstract}



\maketitle

\section{Introduction}

Over a century ago
Frobenius and Schur understood that
the complex representation theory of the symmetric groups $S_n$ for all $n \geq 0$
categorifies
the graded Hopf algebra of symmetric functions 
$\Sym$. The language being used here is much more recent! It means the following.
Let $\Rep(\C S_n)$ denote the category of finite dimensional $\C S_n$-modules
and $S(\lambda)$ be the irreducible Specht module indexed by partition $\lambda \vdash n$.
The isomorphism classes $\{[S(\lambda)]\:|\:\lambda \vdash n\}$ give a distinguished basis for 
Grothendieck group $[\Rep(\C S_n)]$ of this semisimple 
category.
Given an $S_m$-module $V$ and an $S_n$-module $W$, we can form their
induction product $V \circ W := \operatorname{Ind}_{S_m \times S_n}^{S_{m+n}} V \boxtimes W$. This operation descends to the Grothendieck groups to give a multiplication
making 
$$
[\Rep(\C S)] := \bigoplus_{n \geq 0} [\Rep(\C S_n)]
$$
into a graded algebra. 
Moreover the restriction functors $\Res^{S_{m+n}}_{S_m \times S_n}$ for all $m, n \geq 0$ induce a comultiplication 
on $[\Rep(\C S)]$, making it into a graded Hopf algebra.
The categorificiation theorem asserts that it is isomorphic as a graded Hopf algebra to $\Sym$, the canonical isomorphism sending 
$[S(\lambda)]$ to the {\em Schur function}
$s_\lambda \in \Sym$. 

There have been many variations and generalizations of this result since then. 
Perhaps the most relevant for this article comes from the work of Bernstein and Zelevinsky in the late 1970s on the representation theory of
the affine Hecke algebra $H_n$ associated to the general linear group
$GL_n(F)$ over a non-archimedean local field $F$ (e.g. see \cite{Zel1}).
This is even richer algebraically since, unlike $\C S_n$, its finite dimensional representations are no longer completely reducible.
Bernstein and Zelevinsky showed that the direct sum
$$
[\operatorname{Rep}(H)] := \bigoplus_{n \geq 0}
[\operatorname{Rep}(H_n)]
$$
of the Grothendieck groups of the 
categories of finite dimensional $H_n$-modules
for all $n \geq 0$ again has a natural structure of graded Hopf algebra.
Moreover this graded Hopf algebra 
is isomorphic to the coordinate algebra
$\Z[N]$ of a certain direct limit $N$ 
of groups of upper unitriangular matrices.
In \cite{Zel2}, Zelevinsky went on to formulate a $p$-adic analogue of
the Kazhdan-Lusztig conjecture, which was proved by Ginzburg
\cite{CG}.
Zelevinsky's conjecture implies that the basis for
$\Z[N]$ arising from the isomorphism classes of irreducible
$H_n$-modules for all $n \geq 0$ coincides with the Lusztig-Kashiwara
dual canonical basis.
The dual version of this theorem was proved by Ariki in
\cite{Ar}, who
also investigated certain finite dimensional quotients of $H_n$ called
{\em cyclotomic Hecke algebras}, which categorify the integrable
highest weight modules of the corresponding Lie algebra
$\mathfrak{g}
 = \mathfrak{sl}_\infty$.
Ariki's work includes the case that the defining parameter of
$H_n$ is a primitve $p$th root of unity, when $\mathfrak{sl}_\infty$ is replaced by the
affine Kac-Moody algebra
$\widehat{\mathfrak{sl}}_p$.

Quiver Hecke algebras were discovered independently in 2008 
by Khovanov and Lauda \cite{KL1, KL2} and Rouquier \cite{R1}.
They are certain Hecke algebras attached to symmetrizable Cartan
matrices.
It appears that Khovanov and Lauda came upon these algebras 
from an
investigation of endomorphisms of Soergel bimodules (and
related bimodules which arise from cohomology of partial flag varieties), while Rouquier's motivation was a
close analysis of Lusztig's construction of canonical bases
in terms of perverse sheaves on certain quiver varieties. 
In a perfect analogy with the Bernstein-Zelevinsky theory just
described, these algebras categorify the coordinate algebra of
the unipotent group $N$ associated to a maximal nilpotent subalgebra
of the Kac-Moody
algebra $\mathfrak{g}$ arising
from the given Cartan matrix. 

In fact the picture is even better: the
quiver Hecke algebras are naturally $\Z$-graded, so that the
Grothendieck groups of their categories of finite dimensional graded
representations
are $\Z[q,q^{-1}]$-modules, with $q$ acting by degree shift. The
resulting ``graded'' Grothendieck groups categorify a
$\Z[q,q^{-1}]$-form for the quantum group $\mathbf{f}$ that is half of
the quantized enveloping algebra $U_q(\mathfrak{g})$.
Moreover there is an analogue of Ariki's theorem, conjectured
originally by
Khovanov--Lauda and proved by 
Varagnolo-Vasserot \cite{VV} and Rouquier \cite[Corollary 5.8]{R2} using
geometric methods in the spirit of \cite{CG}. 
There are even
cyclotomic quotients of the quiver Hecke algebras which 
Kang-Kashiwara \cite{KK}, Rouquier \cite[Theorem
4.25]{R2} and Webster \cite{W1} have used to categorify 
integrable highest weight modules.
Rouquier also observed in (finite or affine) type A that 
the quiver Hecke algebras become isomorphic to the affine Hecke
algebras discussed earlier (at a generic parameter or a root of unity) when suitably
localized (see \cite[Proposition 3.15]{R1} and also \cite{BK1} in the cyclotomic
setting). 
Thus Ariki's theorem is a special case of the Rouquier-Varagnolo-Vasserot categorificiation
theorem just mentioned (see \cite{BK2}).

Even more variations on the quiver Hecke algebras have subsequently
emerged, including a twisted version related to affine Hecke algebras of type
B introduced by Varagnolo and Vasserot \cite{VV2}, and the quiver Hecke
superalgebras of Kang, Kashiwara and Tsuchioka \cite{KKT}.
The latter 
superalgebras generalize Wang's spin Hecke algebras
\cite{Wang} and the odd nil Hecke algebra of  Ellis, Khovanov and Lauda \cite{EKL},
and give a completely new ``supercategorification'' of the same
quantum groups/highest weight modules as above (see \cite{KKO}). 
We also mention the work \cite{KKK} which connects quiver Hecke
algebras to quantum affine algebras, potentially providing a direct
algebraic link between the categorifications of $\Z[N]$ arising via quiver Hecke algebras
and the ones introduced by Hernandez and Leclerc in \cite{HL}.

For the future perhaps the most exciting development arising from these new algebras is the
introduction again by Khovanov-Lauda \cite{KL3} and Rouquier \cite{R2}
of certain $2$-categories called {\em $2$-Kac-Moody algebras}. These 
categorify Lusztig's idempotented version $\dot U_q(\mathfrak{g})$ of
the quantized enveloping algebra of $\mathfrak{g}$ (see \cite{KL3, W1, W2}). In the case
$\mathfrak{g} = \mathfrak{sl}_2$ this goes back to work of 
Chuang-Rouquier \cite{CR} 
and Lauda \cite{Lauda}.
In the introduction of \cite{R1}, Rouquier promises to define 
a tensor product on the 2-category of dg $2$-representations of the
$2$-Kac-Moody algebra, the ultimate goal being to construct 4-dimensional TQFTs
in fulfillment of predictions made long ago by Crane and Frenkel
\cite{CF}. Webster has also suggested a more down-to-earth
diagrammatic approach to constructing categorifications of tensor products 
of integrable highest weight modules
in finite types in \cite{W1}.

In this article we will not discuss at all any of these higher themes,
aiming instead to give a gentle and self-contained introduction to the 
quiver Hecke algebras and their connection to Lusztig's algebra $\f$,
focussing just on the case of symmetric Cartan matrices for simplicity.
In the last section of the article we specialize further to
finite type and explain some of the interesting homological properties
of quiver Hecke algebras in that setting, similar in spirit to those
of a quasi-hereditary algebra, despite being infinite-dimensional.
As we go we have included proofs or sketch proofs of many of the 
foundational results, before switching into full survey mode later on.
To improve readability,  references to the literature are deferred to the end of each
section. 

\section{Quiver Hecke algebras}

In this opening section, we give a general introduction to the
definition and structure of quiver Hecke algebras. 

\subsection*{Gradings}
Fix once and for all an algebraically closed ground field $\K$.
Everything (vector spaces, algebras, modules, \dots) will be
$\Z$-graded.
For a graded vector space $V = \bigoplus_{n \in \Z} V_n$, 
its graded dimension is
$\Dim V:=\sum_{n \in \Z} (\dim V_n) q^n$, where $q$ is a formal variable.
Of course this only makes sense if $V$ is {\em locally finite
  dimensional}, i.e. the graded pieces of $V$ are finite dimensional.
Typically $V$ will also be {\em bounded below}, i.e. $V_n = 0$ for $n
\ll 0$, in which case $\Dim V$ is a formal Laurent series in $q$.
We write $q^m V$ for the
upward degree shift by $m$ steps, so $q^m V$ is the graded vector space with 
$(q^m V)_n :=
V_{n-m}$, and then $\Dim q^m V = q^m \Dim V$.
More generally, for $f(q) = \sum_{n \in \Z} f_n q^n$,
we write $f(q) V$ for $\bigoplus_{n \in \Z} (q^n V)^{\oplus n}$.
Finally we write $\Hom(V, W)$ for the graded vector space 
$\bigoplus_{n \in \Z} \Hom(V, W)_n$, where $\Hom(V,W)_n$ denotes the 
linear maps $f:V \rightarrow W$ that are homogeneous of degree $n$, 
i.e. they map $V_m$ into $W_{m+n}$.
Note then that $\End(V) := \Hom(V,V)$ is a graded algebra.

\subsection*{Demazure operators}

Recall that the symmetric group $S_n$ is generated by the basic
transpositions
$t_1,\dots,t_{n-1}$ subject only to the {braid relations}
$t_i t_{i+1} t_i =
t_{i+1} t_i t_{i+1}$ and $t_i t_j = t_j t_i$ for $|i-j|>1$, plus
the quadratic relations $t_i^2 = 1$.
The {\em length} $\ell(w)$ of $w \in S_n$ is $\#\{1 \leq i < j \leq
n\:|\:w(i) > w(j)\}$.
We will denote the longest element of $S_n$ by $w_{[1,n]}$. This is
the permutation $1 \mapsto n$, $2 \mapsto n-1$, $3 \mapsto
n-2$, \dots, $n \mapsto 1$, which
is of length $\frac{1}{2}n(n-1)$.
Letting $[n] := \frac{q^n - q^{-n}}{q-q^{-1}}$ and $[n]! := [n] [n-1]
\cdots [2][1]$, we have the well-known factorization of the
Poincar\'e polynomial of $S_n$:
\begin{equation}\label{poincare}
\sum_{w \in S_n} q^{2 \ell(w)} = q^{\frac{1}{2}n(n-1)} [n]!.
\end{equation}

Let $S_n$ act on the polynomial algebra
$\Pol_n := \K[x_1,\dots,x_n]$ by permuting the variables. Viewing $\Pol_n$
as a graded algebra with each $x_i$ in degree $2$, this is an action by graded algebra automorphisms.
So the invariants 
form a graded subalgebra $\Sym_n := \Pol_n^{S_n}$, namely,
the algebra of
{\em symmetric polynomials}. This is again a free polynomial algebra
generated by the 
{elementary symmetric polynomials} 
$e_r := \sum_{1 \leq i_1 < \cdots < i_r\leq n} x_{i_1}\cdots x_{i_r}$
for $r=1,\dots,n$, hence
\begin{equation}\label{symdim}
\Dim \Sym_n = \frac{1}{(1-q^2)(1-q^4)\cdots(1-q^{2n})}.
\end{equation}
For $i=1,\dots,n-1$, we have the {\em Demazure operator}
\begin{equation}
\partial_i:\Pol_n \rightarrow \Pol_n,\qquad
f \mapsto \frac{t_i(f) - f}{x_i - x_{i+1}}.
\end{equation}
This is a homogeneous linear map of degree $-2$
such that $\partial_i(fg) = \partial_i(f) g + t_i(f) \partial_i(g)$.
From this identity, it is easy to see that $\partial_i$ is a
$\Sym_n$-module homomorphism.
The endomorphisms
$\partial_1,\dots,\partial_{n-1}$ satisfy the same braid relations
as in the symmetric group, hence for each $w \in S_n$
there is a well-defined operator
$\partial_w \in \End(\Pol_n)_{-2\ell(w)}$ such that
$\partial_w = \partial_{i_1} \cdots \partial_{i_k}$
if $w = t_{i_1}\cdots t_{i_k}$ is a 
{\em reduced expression} for $w$, i.e. $k=\ell(w)$.
Moreover we have that $\partial_i^2 = 0$ for each $i=1,\dots,n-1$.

\begin{Theorem}\label{dem}
The polynomial algebra $\Pol_n$ is a free 
$\Sym_n$-module of rank $n!$, with basis $\left(b_w\right)_{w \in S_n}$
defined from $b_w := \partial_w (x_2 x_3^2 \cdots x_n^{n-1})$.
Each $b_w$ is homogeneous of degree $n(n-1) - 2 \ell(w)$, and $b_{w_{[1,n]}} = 1$.
\end{Theorem}

\begin{proof}
We first show by induction on $n$ that $b_{w_{[1,n]}} = 1$.
Let $w_{[2,n]}$ denote 
the longest element of $S_{n-1}$ embedded into $S_n$ as the permutations fixing $1$,
so that $w_{[1,n]} = t_{n-1}\cdots t_1 w_{[2,n]}$.
Then:
\begin{align*}
\partial_{w_{[1,n]}}(x_2 x_3^2 \cdots x_n^{n-1})
&= \partial_{n-1} \cdots \partial_1 \partial_{w_{[2,n]}}
((x_2\cdots x_n)(x_3\cdots x_n^{n-2}))\\
&=\partial_{n-1}\cdots\partial_1 (x_2 \cdots x_n \partial_{w_{[2,n]}}(x_3 \cdots x_n^{n-2}))= 1.
\end{align*}
Now we use this to show that the elements $(b_w)_{w \in S_n}$ are $\Sym_n$-linearly independent. Suppose that $\sum_{w \in S_n} p_w b_w = 0$ for some $p_w \in \Sym_n$, not all of which are zero.
Let $w \in S_n$ 
be of minimal length such that $p_w \neq 0$, and write $w_{[1,n]} = w' w$ for
$w' \in S_n$.
Then apply $\partial_{w'}$ to the identity $\sum_{w \in S_n} p_w \partial_w(x_2 x_3^2 \cdots x_n^{n-1}) = 0$ to deduce that
$p_w = 0$, a contradiction.
Finally to complete the proof we check graded dimensions:
\begin{align*}
\Dim \left(\bigoplus_{w \in S_n} \Sym_n b_w \right)&\!\stackrel{(\ref{symdim})}{=}\!
\frac{\sum_{w \in S_n}  q^{n(n-1)-2\ell(w)}}{(1-q^2)(1-q^4)\cdots (1-q^{2n})}
\!\stackrel{(\ref{poincare})}{=}\!\frac{1}{(1-q^2)^n}=\Dim \Pol_n.
\end{align*}
 \end{proof}

\begin{Corollary}\label{mx}
The endomorphism algebra $\End_{\Sym_n}(\Pol_n)$ is isomorphic to the algebra of 
$n! \times n!$ matrices with entries in $\Sym_n$.
More precisely, its center is identified with $\Sym_n$, it is free as a module over its center with a 
basis of matrix units
$(e_{x,y})_{x,y \in S_n}$ defined from
$e_{x,y}(b_w) := \delta_{y,w} b_x$,
and each $e_{x,y}$ is homogeneous of degree $2(\ell(y)-\ell(x))$.
\end{Corollary}

\subsection*{The nil Hecke algebra}
The nil Hecke algebra $\NH_n$ is the quiver Hecke algebra (to be defined formally in the next subsection)
for the trivial quiver $\bullet$. By definition, it is the associative graded 
$\K$-algebra
with homogeneous generators $x_1,\dots,x_n$ of degree $2$ and
$\tau_1,\dots,\tau_{n-1}$ of degree $-2$, subject to the following relations:
the $x_i$'s commute,
the $\tau_i$'s satisfy the same braid relations as in the symmetric group
plus the quadratic relations
$\tau_i^2 = 0$, and finally
$\tau_i x_j - x_{t_i(j)} \tau_i =
\delta_{i+1,j}-\delta_{i,j}.
$
As the $\tau_i$'s satisfy the braid relations, there are well-defined
elements $\tau_w \in \NH_n$ for each $w \in S_n$ such that
$\tau_w = \tau_{i_1}\cdots\tau_{i_k}$ whenever $w = t_{i_1}\cdots t_{i_k}$ is a reduced expression.
The definition of $\NH_n$ ensures that we can make the polynomial algebra $\Pol_n$ into a left $\NH_n$-module by declaring that each $x_i$ acts by left multiplication
and each $\tau_i$ acts by the Demazure operator $\partial_i$.

\begin{Theorem}\label{nil}
The nil Hecke algebra $\NH_n$ has basis
$$\{x_1^{m_1} \cdots x_n^{m_n}\tau_w\:|\:w \in S_n,m_1,\dots,m_n \geq 0\}.$$
Moreover the action of $\NH_n$ on $\Pol_n$ induces a graded algebra isomorphism
$$
\NH_n \stackrel{\sim}{\rightarrow}
\End_{\Sym_n}(\Pol_n).
$$
\end{Theorem}

\begin{proof}
It is clear from the relations that the given monomials span $\NH_n$.
To show that they are linearly independent, 
suppose there is a non-trivial linear relation $\sum_{w \in S_n} \sum_{m_1,\dots,m_n \geq 0} c_{w,m_1,\dots,m_n} x_1^{m_1}\cdots x_n^{m_n} \tau_w = 0$.
Let $w$ be of minimal length such that $c_{w,m_1,\dots,m_n} \neq 0$
for some $m_1,\dots,m_n$.
Write $w_{[1,n]} = w w'$ then act on the vector $b_{w'}$ from Theorem~\ref{dem}
to obtain the desired contradiction.
This argument shows in fact that the homomorphism $\NH_n \rightarrow \End_{\Sym_n}(\Pol_n)$ induced by the action of $\NH_n$ on $\Pol_n$
is injective. Finally it is surjective by a graded dimension calculation.
\end{proof}

The theorem shows in particular that the algebra $\Pol_n$ embeds into $\NH_n$
as the subalgebra generated by $x_1,\dots,x_n$. 
Using also Corollary~\ref{mx},
we deduce that $Z(\NH_n) = \Sym_n$, and $\NH_n$ is isomorphic to the algebra of
$n! \times n!$ matrices over its center.
Let
\begin{equation}\label{en}
e_n := x_2 x_3^2\cdots x_n^{n-1} \tau_{w_{[1,n]}}.
\end{equation}
Recalling Theorem~\ref{dem}, 
we have that
$e_n b_1 = b_1$ and $e_n b_w = 0$ for $w \neq 1$.
Hence $e_n$ corresponds under the isomorphism from Theorem~\ref{nil} to the matrix unit $e_{1,1}$ of Corollary~\ref{mx}, so it is a primitive idempotent.
It follows that
\begin{equation}
P_n := q^{\frac{1}{2}n(n-1)}\NH_n e_n
\end{equation}
is an indecomposable projective module.
It has a unique irreducible
{\em graded} quotient
which we denote by $L_n$.

\begin{Corollary}\label{ln}
The left regular module $\NH_n$ is isomorphic to 
$[n]! P_n$ as a graded module.
Hence $\Dim L_n = [n]!$.
\end{Corollary}

\begin{proof}
Using Theorem~\ref{nil}, we identify $\NH_n$ with 
$\End_{\Sym_n}(\Pol_n)$.
Then as in Corollary~\ref{mx}
we have that $\NH_n = \bigoplus_{w \in S_n} \NH_n e_{w,w}$.
Right multiplication by $e_{w,1}$ is an isomorphism of graded modules
$\NH_n e_{w,w} \cong q^{2\ell(w)} \NH_n e_{1,1}$.
Thus 
$$
\NH_n \cong \bigoplus_{w \in S_n} q^{2 \ell(w)} \NH_n e_n
\stackrel{(\ref{poincare})}{\cong} [n]! \left(q^{\frac{1}{2}n(n-1)} \NH_n e_n\right)
= [n]! P_n.
$$
Finally 
$\Dim L_n = \Dim \Hom_{\NH_n}(\NH_n, L_n)
= [n]! \Dim \Hom_{\NH_n}(P_n, L_n) = [n]!$.
\end{proof}

\begin{Corollary}
$P_n\cong q^{-\frac{1}{2}n(n-1)} \Pol_n$.
\end{Corollary}

\begin{proof}
It is clear from Theorems~\ref{nil} and \ref{dem} that $\Pol_n$ is a 
projective indecomposable $\NH_n$-module, so it is isomorphic to $P_n$ up to a degree shift. To determine the degree shift, compare graded dimensions.
\end{proof}

\subsection*{The quiver Hecke algebra}
Fix now a loop-free quiver with finite vertex set $I$.
For $i,j \in I$, 
let $m_{i,j}$ denote the number of directed edges $i \rightarrow j$.
The corresponding (symmetric) {\em Cartan matrix} $C = (c_{i,j})_{i,j
  \in I}$ is defined from
$c_{i,i} := 2$ and $c_{i,j} := -m_{i,j}-m_{j,i}$ for $i \neq j$.
To $C$ there is an associated (symmetric) Kac-Moody algebra $\mathfrak{g}$.
We fix a choice of root datum for 
$\mathfrak{g}$. This gives us a {\em weight lattice} $P$, which is a
finitely generated free abelian group equipped with a symmetric bilinear form $P \times P \rightarrow \Q, 
(\lambda,\mu) \mapsto \lambda\cdot\mu$, containing
{\em simple roots}
$(\alpha_i)_{i \in I}$ and {\em fundamental weights} $(\varpi_i)_{i \in I}$
such that $\alpha_i \cdot \alpha_j = c_{i,j}$ and $\alpha_i \cdot \omega_j = \delta_{i,j}$ for all $i,j \in I$.
The {\em root lattice} is
$Q := \bigoplus_{i \in I} \Z \alpha_i \subset P$.
Also let $Q^+ := \bigoplus_{i \in I} \N \alpha_i \subset Q$, and define
the {\em height} of $\alpha = \sum_{i \in I} c_i \alpha_i \in Q^+$
to be the sum $\height(\alpha) := \sum_{i \in I} c_i$ of its coefficients.
Finally let $\W$ denote the set of all words 
in the alphabet $I$ and, for $\alpha \in Q^+$ of height $n$, 
let $\W_\alpha \subset \W$
denote the words $\bi = i_1\cdots i_n \in \W$ such that $\alpha_{i_1}+\cdots+\alpha_{i_n} = \alpha$.
The symmetric group $S_n$ acts on $\W_\alpha$ by permuting letters in the obvious way.

Let $q_{i,j}(u,v) \in \K[u,v]$ denote $0$ if $i = j$ or 
$(v-u)^{m_{i,j}}(u-v)^{m_{j,i}}$ if $i \neq j$.
For $\alpha \in Q^+$ of height $n$, the {\em quiver Hecke algebra} $H_\alpha$ is the associative $\K$-algebra on generators
$
\{1_\bi\:|\:\bi \in \W_\alpha\}\cup\{x_1,\dots,x_n\}\cup\{\tau_1,\dots,\tau_{n-1}\}
$
subject to the following relations:
\begin{itemize}
\item the $1_\bi$'s are orthogonal idempotents summing to the identity $1_\alpha \in H_\alpha$;
\item $1_\bi x_k = x_k 1_\bi$ and $1_\bi \tau_k = \tau_k 1_{t_k(\bi)}$;
\item $x_1,\dots,x_n$ commute;
\item $(\tau_k x_l -x_{t_k(l)} \tau_k)1_\bi
=\delta_{i_k, i_{k+1}}(\delta_{k+1,l} - \delta_{k,l})1_\bi;
$
\item 
$\tau_k^2 1_\bi = q_{i_k, i_{k+1}}(x_k, x_{k+1}) 1_\bi$
\item $\tau_k \tau_l = \tau_l \tau_k$ if $|k-l| > 1$;
\item
$\left(\tau_{k+1} \tau_k \tau_{k+1}\!-\!\tau_k \tau_{k+1} \tau_k\right)1_\bi
= 
\delta_{i_k, i_{k+2}}
\!\!\displaystyle\frac{q_{i_k, i_{k+1}}\!(x_k, x_{k+1}) \!-\! q_{i_k,
    i_{k+1}}\!(x_{k+2},x_{k+1})}{x_k - x_{k+2}} 1_\bi.$
\end{itemize}
There is a well-defined $\Z$-grading on $H_\alpha$ such that each $1_\bi$ is of degree 0,
each $x_j$ is of degree $2$, and each $\tau_k 1_\bi$ is of degree
$-\alpha_{i_k} \cdot \alpha_{i_{k+1}}$.

Note right away that if $\alpha = n \alpha_i$ for $i
\in I$, then the quiver Hecke algebra
$H_{n \alpha_i}$ is just a copy of the nil Hecke algebra
$\NH_n$. In particular there is just one irreducible graded
left $H_{n\alpha_i}$-module
up to isomorphism and degree shift. A representative for it 
may be constructed as the irreducible head
$L(i^n)$ of the projective indecomposable module
\begin{equation}\label{pin}
P(i^n) := q^{\frac{1}{2}n(n-1)}H_{n \alpha_i} e_n
\end{equation} where $e_n \in H_{n\alpha_i}$ is the
primitive idempotent defined like in (\ref{en}).
By Corollary~\ref{ln}, $L(i^n)$ has graded dimension $[n]!$.

It is common---and convenient for calculations---to interpret $H_\alpha$
 diagrammatically. In this paragraph we explain this point of view
 under the simplifying assumption that the underlying quiver is {\em
   simply-laced}, i.e. $m_{i,j} + m_{j,i} \leq 1$ for all $i \neq j$.
Start with the free graded $\K$-linear monoidal category $H'$
generated by objects $i\:\:(i \in I)$ and homogeneous morphisms
$x:i \rightarrow i$ of degree 2 and $\tau:ij \rightarrow ji$ of degree
$-\alpha_i \cdot \alpha_j$
for all $i,j \in I$.
We represent $x$ and $\tau$ diagrammatically by:
$$
\begin{picture}(120,30)
\put(-20,9){$x=$}
\put(2.2,1){\line(0,1){20}}
\put(-0.4,9){$\bullet$}
\put(0.4,-3){$_i$}
\put(0.4,25){$_i$}
\end{picture}
\begin{picture}(20,30)
\put(-20,9){$\tau=$}
\put(2.2,1){\line(1,2){10}}
\put(12.2,1){\line(-1,2){10}}
\put(0,-3){$_i$}
\put(10,-3){$_j$}
\put(0.4,25){$_j$}
\put(10.4,25){$_i$}
\end{picture}
$$
Composition of morphisms 
corresponds to 
vertical concatenation of diagrams (so $a \circ b$ is the diagram $a$ on top of the diagram $b$),
while tensor product is horizontal concatenation (so $a \otimes b$ is $a$ to the left of $b$).
Arbitrary objects are 
tensor products $i_1 \otimes \cdots \otimes i_n$
of the generators $i_1,\dots,i_n \in I$, which we identify with
words $\bi = i_1 \cdots i_n \in \W$. Then, for two words $\bi,\bj \in \W$, an arbitrary
morphism $\bi \rightarrow \bj$ is a linear combination
of diagrams obtained by composing $x$'s and $\tau$'s horizontally and vertically
to obtain braid-like diagrams
with strings consistently colored by the letters of the 
word $\bi$ at the bottom and the word $\bj$ at the top.
In particular, $\Hom_{H'}(\bi, \bj) = \varnothing$ unless $\bi$ and $\bj$ both lie in $\W_\alpha$ for some $\alpha \in Q^+$.
Then the {\em quiver Hecke category} $H$ is the $\K$-linear monoidal category obtained from $H'$ by 
imposing the following relations:
$$
\begin{picture}(65,27)
\put(2.2,6){\line(1,2){10}}
\put(12.2,6){\line(-1,2){10}}
\put(0,2){$_i$}
\put(3.7,2){$_{\neq}$}
\put(10,2){$_j$}
\put(7.8,19.5){$\bullet$}
\put(18,14){$=$}
\put(32.2,6){\line(1,2){10}}
\put(42.2,6){\line(-1,2){10}}
\put(30,2){$_i$}
\put(33.7,2){$_{\neq}$}
\put(40,2){$_j$}
\put(32,7.9){$\bullet$}
\end{picture}
\begin{picture}(65,27)
\put(2.2,6){\line(1,2){10}}
\put(12.2,6){\line(-1,2){10}}
\put(0,2){$_i$}
\put(3.7,2){$_{\neq}$}
\put(10,2){$_j$}
\put(7.4,7.9){$\bullet$}
\put(18,14){$=$}
\put(32.2,6){\line(1,2){10}}
\put(42.2,6){\line(-1,2){10}}
\put(30,2){$_i$}
\put(33.7,2){$_{\neq}$}
\put(40,2){$_j$}
\put(31.8,19.5){$\bullet$}
\end{picture}
\begin{picture}(95,27)
\put(2.2,6){\line(1,2){10}}
\put(12.2,6){\line(-1,2){10}}
\put(0,2){$_i$}
\put(10,2){$_i$}
\put(7.8,19.5){$\bullet$}
\put(18,14){$=$}
\put(32.2,6){\line(1,2){10}}
\put(42.2,6){\line(-1,2){10}}
\put(30,2){$_i$}
\put(40,2){$_i$}
\put(32,7.9){$\bullet$}
\put(47,14){$+$}
\put(72.2,6){\line(0,1){20}}
\put(62.2,6){\line(0,1){20}}
\put(60,2){$_i$}
\put(70,2){$_i$}
\end{picture}
\begin{picture}(80,27)
\put(2.2,6){\line(1,2){10}}
\put(12.2,6){\line(-1,2){10}}
\put(0,2){$_i$}
\put(10,2){$_i$}
\put(7.4,7.9){$\bullet$}
\put(18,14){$=$}
\put(32.2,6){\line(1,2){10}}
\put(42.2,6){\line(-1,2){10}}
\put(30,2){$_i$}
\put(40,2){$_i$}
\put(31.8,19.5){$\bullet$}
\put(47,14){$+$}
\put(72.2,6){\line(0,1){20}}
\put(62.2,6){\line(0,1){20}}
\put(60,2){$_i$}
\put(70,2){$_i$}
\end{picture}
$$
$$
\:\:\begin{picture}(15,15)
\qbezier(2.2,-6)(18.2,4)(2.2,14)
\qbezier(12.2,-6)(-3.8,4)(12.2,14)
\put(0,-10){$_i$}
\put(10,-10){$_j$}
\end{picture}
=\left\{\begin{array}{cl}
0&\text{if $i=j$}\\
\begin{picture}(15,18)
\put(2.2,-6){\line(0,1){20}}
\put(12.2,-6){\line(0,1){20}}
\put(0,-10){$_i$}
\put(10,-10){$_j$}
\put(9.9,0.4){$\bullet$}
\end{picture}
-\begin{picture}(15,18)
\put(-0.1,0.4){$\bullet$}
\put(2.2,-6){\line(0,1){20}}
\put(12.2,-6){\line(0,1){20}}
\put(0,-10){$_i$}
\put(10,-10){$_j$}
\end{picture}
&\text{if $i \rightarrow j$}\\
\begin{picture}(15,30)
\put(2.2,-6){\line(0,1){20}}
\put(12.2,-6){\line(0,1){20}}
\put(0,-10){$_i$}
\put(10,-10){$_j$}
\put(-0.1,0.4){$\bullet$}
\end{picture}
-\begin{picture}(15,30)
\put(9.9,0.4){$\bullet$}
\put(2.2,-6){\line(0,1){20}}
\put(12.2,-6){\line(0,1){20}}
\put(0,-10){$_i$}
\put(10,-10){$_j$}
\end{picture}
&\text{if $i \leftarrow j$}\\
\begin{picture}(15,30)
\put(2.2,-6){\line(0,1){20}}
\put(12.2,-6){\line(0,1){20}}
\put(0,-10){$_i$}
\put(10,-10){$_j$}
\end{picture}
&\text{otherwise,}\\\\
\end{array}\right.
\qquad
\begin{picture}(24,30)
\put(2.2,-11){\line(2,3){20}}
\qbezier(12.2,-11)(26.6,4)(12.2,19)
\put(22.2,-11){\line(-2,3){20}}
\put(0,-15){$_i$}
\put(10,-15){$_j$}
\put(20,-15){$_k$}
\end{picture}
-
\begin{picture}(25,30)
\put(2.2,-11){\line(2,3){20}}
\qbezier(12.2,-11)(-2,4)(12.2,19)
\put(22.2,-11){\line(-2,3){20}}
\put(0,-15){$_i$}
\put(10,-15){$_j$}
\put(20,-15){$_k$}
\end{picture}
=
\left\{\:\:
\begin{array}{cl}
\begin{picture}(25,20)
\put(-8,0){$-$}
\put(2.2,-11){\line(0,1){30}}
\put(12.2,-11){\line(0,1){30}}
\put(22.2,-11){\line(0,1){30}}
\put(0,-15){$_i$}
\put(10,-15){$_j$}
\put(20,-15){$_k$}
\end{picture}
&\text{if $i=k \rightarrow j$}\\
\begin{picture}(25,40)
\put(2.2,-11){\line(0,1){30}}
\put(12.2,-11){\line(0,1){30}}
\put(22.2,-11){\line(0,1){30}}
\put(0,-15){$_i$}
\put(10,-15){$_j$}
\put(20,-15){$_k$}
\end{picture}
&\text{if $i=k \leftarrow j$}\\\\\\
0&\text{otherwise.}
\end{array}\right.
$$
The quiver Hecke algebra $H_\alpha = \bigoplus_{\bi,\bj \in \W_\alpha} 1_\bj H_\alpha 1_\bi$ from before is identified with the vector space
$\bigoplus_{\bi,\bj \in \W_\alpha} \Hom_H(\bi, \bj)$, so that multiplication in $H_\alpha$ corresponds to vertical composition of morphisms in $H$.

For $\alpha \in Q^+$ of height $n$ again, 
the relations imply that there is 
an antiautomorphism $\t:H_\alpha \rightarrow H_\alpha$
defined on generators by
\begin{align}
\t(1_\bi) &= 1_\bi,
&\t(x_k) &= x_k,
&\t(\tau_k) &= \tau_k.
\end{align}
In diagrammatic terms, 
$\t$ reflects in a horizontal axis.

\subsection*{Basis theorem and center}
Suppose in this subsection that $\alpha \in Q^+$ is of height $n$.
In general the braid relations are only approximately true in $H_\alpha$.
So, to write down a basis, 
we must fix a choice of a distinguished reduced expression
$w = t_{i_1}\cdots t_{i_k}$ for each $w \in S_n$, then 
{\em define} $\tau_w := \tau_{i_1}\cdots \tau_{i_k} \in H_\alpha$.
Also, for $w \in S_n$ and $\bi \in \W_\alpha$, let
\begin{equation}\label{degdef}
\deg(w;\bi) := -\displaystyle\sum_{\substack{1 \leq j < k \leq n\\ w(j) > w(k)}}
\alpha_{i_j}\cdot\alpha_{i_k}.
\end{equation}
The following theorem is proved in a similar way to 
the special case Theorem~\ref{nil}, by constructing
a certain {\em polynomial representation} of the quiver Hecke algebra
on the underlying vector space $\Pol_\alpha := \bigoplus_{\bi \in \W_\alpha}
\Pol_n 1_\bi$.

\begin{Theorem}[``Basis theorem'']\label{qhbasis}
The monomials
$$
\{x_1^{m_1}\cdots x_n^{m_n} \tau_w 1_\bi\:|\:\bi \in \W_\alpha, w \in S_n,
m_1,\dots,m_n \geq 0\}
$$
give a basis for $H_\alpha$.
Hence, for $\bi,\bj \in \W_\alpha$, we have that
\begin{equation}\label{iden}
\Dim 1_\bj H_\alpha 1_\bi = 
\frac{1}{(1-q^2)^n}\sum_{\substack{w \in S_n \\ w(\bi) = \bj}} q^{\deg(w;\bi)}.
\end{equation}
\end{Theorem}

Note Theorem~\ref{qhbasis} shows in particular that $H_\alpha$ is locally finite dimensional and bounded below. 
The final basic result in this section is concerned with the center $Z(H_\alpha)$.
To formulate it, we pick $\bi \in \W_\alpha$ so that $S_\bi := \operatorname{Stab}_{S_n}(\bi)$ is a
standard parabolic subgroup of $S_n$, i.e. all equal letters in the word $\bi$ appear consecutively.
For $j=1,\dots,n$, let
\begin{equation}\label{centraldef}
z_j := \sum_{w \in S_n / S_\bi} x_{w(j)} 1_{w(\bi)},
\end{equation}
where $S_n / S_\bi$ denotes the set of minimal length left coset
representatives.
In view of Theorem~\ref{qhbasis}, these elements generate a free polynomial algebra
$\K[z_1,\dots,z_n]$ inside $H_\alpha$.
We let $S_\bi \leq S_n$ act on $\K[z_1,\dots,z_n]$ by
permuting the generators.

\begin{Theorem}[``Center'']\label{center}
We have that
$$
Z(H_\alpha)=\K[z_1,\dots,z_n]^{S_\bi}.
$$
Hence $H_\alpha$ is free of
finite rank as a module over its center; forgetting the grading the
rank is $(n!)^2$.
\end{Theorem}

To illustrate how the second statement of the theorem is deduced from
the first, consider the special case $I = \{1,2\}$ and $\alpha = 2 \alpha_1 + \alpha_2$,
and take $\bi = 112$. Then
\begin{align*}
z_1 &= x_1 1_{112} + x_1 1_{121} + x_2 1_{211},\\
z_2 &= x_2 1_{112} + x_3 1_{121}+x_3 1_{211},\\
z_3 &= x_3 1_{112} + x_2 1_{121} + x_1 1_{211}.
\end{align*}
By the first part of the theorem, $Z(H_\alpha)$ is freely generated by the elements
$z_1+z_2$, $z_1z_2$ and $z_3$.
The algebra $\K[z_1,z_2,z_3]$ is free with basis $\{1, z_1\}$ as a $Z(H_\alpha)$-module;
the algebra $\Pol_\alpha$ embedded in the natural way into $H_\alpha$
is free as a $\K[z_1,z_2,z_3]$-module 
with basis $\{1_{112}, 1_{121}, 1_{211}\}$; and finally $H_\alpha$ is a free
left $\Pol_\alpha$-module 
on basis $\{\tau_w\:|\:w \in S_3\}$. 
Putting it all together we see that $H_\alpha$ is a free $Z(H_\alpha)$-module
of rank 36.

\subsection*{Notes}
Our discussion of nil Hecke algebras follows \cite[$\S$2]{R2}.

Quiver Hecke algebras were introduced 
by Khovanov and Lauda \cite{KL1,KL2} and independently 
by Rouquier \cite{R1} (in a slightly more general form); 
consequently they are also often called {\em Khovanov-Lauda-Rouquier algebras}. 
We have restricted to symmetric Cartan matrices for simplicity, but note that
all of the definitions and results described in this article can be extended to arbitrary symmetrizable Cartan matrices,
with the notable exception of Theorem~\ref{canb}.
The precise normalization of the relations chosen here matches that of \cite[$\S$5]{R2} and \cite{VV}, where 
quiver Hecke algebras for symmetric Cartan matrices
are realized geometrically as ext-algebras of a certain direct sum of degree-shifted irreducible perverse sheaves on a quiver variety.

Theorems~\ref{qhbasis} and \ref{center} are proved both in 
\cite[Theorem 3.7 and Proposition 3.9]{R1} and in \cite[Theorem 2.5 and Theorem 2.9]{KL1}.

\section{Categorification}

Now we describe the main results relating the representation theory of
the quiver Hecke algebras $H_\alpha$ (or the quiver Hecke category
$H$) to Lusztig's algebra $\f$,
i.e. half of the quantized enveloping algebra associated
to the Kac-Moody algebra $\mathfrak{g}$.

\subsection*{Rep and Proj}
By an {\em $H$-module} we henceforth mean a $Q^+$-graded vector
space $V = \bigoplus_{\alpha \in Q^+} 1_\alpha V$ such that 
each $1_\alpha V$ is a {\em graded} left $H_\alpha$-module.
Occasionally we talk about {\em ungraded} modules, meaning each
$1_\alpha V$ is a 
left $H_\alpha$-module without any prescribed grading.
In view of Theorem~\ref{qhbasis}, each $H_\alpha$ is locally finite
dimensional and bounded below, hence irreducible $H$-modules
are automatically finite dimensional.
Also Theorem~\ref{center} implies that 
there are only finitely many
irreducible graded $H_\alpha$-modules for
each $\alpha\in Q^+$
up to isomorphism and degree shift.

\begin{Lemma}\label{abs}
Every irreducible $H$-module remains irreducible as an ungraded module.
If $L_1$ and $L_2$ are two irreducible $H$-modules which are
isomorphic as ungraded modules then there exists a unique $m \in \Z$
such that 
$L_1 \cong q^m L_2$ as graded modules.
\end{Lemma}

\begin{proof}
The first part is \cite[Theorem 4.4.4(v)]{NV}, and the second part follows from \cite[Lemma 2.5.3]{BGS}.
\end{proof}

We have that
$1_\alpha = \bigoplus_{\bi \in \W_\alpha} 1_\bi$ as a sum of mutually
orthogonal idempotents, hence any $H$-module $V$ decomposes as a vector
space as
$V = \bigoplus_{\bi \in \W} 1_\bi V.$
We refer to the subspace $1_\bi V$ as the {\em $\bi$-word space} of $V$.
The {\em character} of a finite dimensional $H$-module $V$ is
defined from
\begin{equation}\label{char}
\Ch V := \sum_{\bi \in \W} (\Dim 1_\bi V) \bi,
\end{equation}
which is an element of the free $\Z[q,q^{-1}]$-module $\Z[q,q^{-1}]\W$
on basis $\W$. 
Forgetting the grading, we also have the {\em ungraded character} $\ch V := \sum_{\bi \in \W} (\dim 1_\bi V)
\bi$, which is the specialization of $\Ch
V$ at $q=1$.

\begin{Lemma}\label{parity}
Fix a total order $<$ on $I$. For a word $\bi \in \W$ of length $n$ define
$$
p(\bi) := \sum_{\substack{1 \leq j < k \leq n\\i_j < i_k}} \alpha_{i_j}\cdot\alpha_{i_k} \pmod{2}.
$$
Then every $H$-module $V$ decomposes as a direct sum of 
modules
 as $V = V^{\bar 0}\oplus V^{\bar 1}$ where $$
V^{q} := \bigoplus_{\substack{\bi \in \W, n \in \Z \\n \equiv p(\bi)+q\: (\operatorname{mod} 2)}}
1_\bi V_n.
$$
\end{Lemma}

\begin{proof}
We just need to check for $\alpha \in Q^+$ 
that $1_\alpha V^{\bar 0}$ and $1_\alpha V^{\bar 1}$ are stable
under the action of the generators of $H_\alpha$.
This is clear for the generators $1_\bi$ and $x_j$ since these are
even and preserve the word spaces. For $\tau_k$, note that it maps
$1_\bi V_n$ into $1_{t_k(\bi)} V_{n-\alpha_{i_k}\cdot \alpha_{i_{k+1}}}$, and we have
that
$p(t_k(\bi)) \equiv p(\bi) -\alpha_{i_k}\cdot\alpha_{i_{k+1}} \pmod{2}$.
\end{proof}

We are going to focus now on the categories
$$
\Rep(H) = \bigoplus_{\alpha \in Q^+} \Rep(H_\alpha),
\qquad
\Proj(H) = \bigoplus_{\alpha \in Q^+} \Proj(H_\alpha)
$$
of finite dimensional $H$-modules and finitely generated projective
$H$-modules, respectively.
Morphisms in both categories are module homomorphisms that are
homogeneous of degree 0. In particular, this ensures that $\Rep(H)$ is
an abelian category. 
We continue to write $\Hom_H(V, W)$ for $\bigoplus_{n \in \Z} \Hom_H(V, W)_n$,
where $\Hom_H(V,W)_n$ means homomorphisms that are homogeneous of degree $n$.
In other words, we are viewing $\Rep(H)$ and $\Proj(H)$ as 
{\em graded categories} equipped with the distinguished
{\em degree shift automorphism} $V \mapsto qV$, and $\Hom_H(V, W)$ is the graded vector space
$\bigoplus_{n \in \Z} \operatorname{hom}(q^n V, W)$ defined via the
internal hom.

The following theorem is fundamental. Its proof is essentially the
same as the proof of the analogous statement for the affine Hecke
algebra $H_n$ from the introduction, which is classical. The argument
uses the Shuffle Lemma (formulated as Corollary~\ref{shuffle} below)
and the properties of the nil Hecke algebra proved already.

\begin{Theorem}[``Linear independence of characters'']\label{chinj}
For $\alpha \in Q^+$,
suppose that $L_1,\dots,L_r$ are representatives for the irreducible
graded left $H_\alpha$-modules up to isomorphism and degree shift.
Their ungraded characters $\ch L_1,\dots, \ch L_r \in \Z \W$ are $\Z$-linearly
independent.
\end{Theorem}

Let $[\Rep(H)]$ denote the Grothendieck group of the abelian category
$\Rep(H)$ and $[\Proj(H)]$ be the split Grothendieck group of the
additive category $\Proj(H)$. 
The isomorphism classes of
irreducible $H$-modules give a basis for $[\Rep(H)]$,
and the isomorphism classes of
projective indecomposable $H$-modules
give a basis for $[\Proj(H)]$, both as free $\Z$-modules.
We often view $[\Rep(H)]$ and $[\Proj(H)]$
as
$\Z[q,q^{-1}]$-modules, with action of $q$ induced by the degree shift
automorphism.
 The character map from (\ref{char}) induces a
$\Z[q,q^{-1}]$-linear map
$$
\Ch:[\Rep(H)] \rightarrow \Z[q,q^{-1}] \W.
$$

\begin{Corollary}\label{chinjc}
The map $\Ch:[\Rep(H)] \rightarrow \Z[q,q^{-1}]\W$ is injective.
\end{Corollary}

\begin{proof}
Let $L_1,\dots,L_r$ be as in Theorem~\ref{chinj}, so that
$[L_1],\dots,[L_r]$ give a basis for $[\Rep(H)]$ as a free
$\Z[q,q^{-1}]$-module. We need to show that their characters
$\Ch L_1,\dots,\Ch L_r$ are $\Z[q,q^{-1}]$-linearly independent.
Take a non-trivial linear relation
$\sum_{i=1}^r f_i(q) \Ch L_i = 0$ for $f_i(q) \in \Z[q,q^{-1}]$. Dividing by
$q-1$ if necessary we may assume that $f_i(1) \neq 0$ for some $i$. 
Then we specialize at $q=1$ to deduce that $\sum_{i=1}^r f_i(1) \ch
L_i = 0$, 
contradicting Theorem~\ref{chinj}.
\end{proof}

There are dualities $\circledast$ and $\#$ on $\Rep(H)$ and
$\Proj(H)$, respectively,
defined by $V^\circledast := \Hom_\K(V, \K)$ and
$P^\# := \Hom_H(P, H)$, with the left action of $H$ arising via the antiautomorphism $\t$.
These dualities induce involutions of the Grothendieck groups
$[\Rep(H)]$ and $[\Proj(H)]$ which are antilinear with respect to the
bar involution
$-:\Z[q,q^{-1}] \rightarrow \Z[q,q^{-1}],
q \mapsto q^{-1}.$
There is also a non-degenerate bilinear pairing
$(\cdot,\cdot):[\Proj(H)] \times [\Rep(H)] \rightarrow \Z[q,q^{-1}]$
defined from
\begin{equation}\label{pairing}
([P], [V]) := 
\Dim \Hom_H(P^\#, V)
= \overline{\Dim \Hom_H(P, V^\circledast)}.
\end{equation}

\begin{Lemma}\label{shift}
Any irreducible $H$-module $L$ can be shifted uniquely in degree so that it
becomes $\circledast$-self-dual. In that case its projective cover
$P$ is $\#$-self-dual.
\end{Lemma}

\begin{proof}
Let $L$ be an irreducible $H$-module.
It is immediate from the definition of $\circledast$ that the word
spaces of $L$ and $L^\circledast$ have the same dimensions (although
they might not have the same graded dimensions).
In other words, we have that 
$\ch L = \ch (L^\circledast)$.
Combined with Theorem~\ref{chinj} and Lemma~\ref{abs}, this means that 
there is a unique $m \in \Z$ such that $L^\circledast\cong q^m L$.
Now pick $\bi$ such that $1_\bi L \neq 0$.
By Lemma~\ref{parity} we have that $$
\Dim 1_\bi L = a_0 q^p + a_1 q^{p+2} +
\cdots + a_k q^{p+2k}
$$ for some $p \in \Z$, $k \geq 0$ and 
$a_0,\dots,a_k \in \N$ with $a_0, a_k \neq 0$.
Hence $$
\Dim 1_\bi L^\circledast = q^m (a_0 q^p+\cdots+a_k q^{p+2k})
= a_k q^{-p-2k}+\cdots+a_0 q^{-p}.
$$
We deduce that $m+p=-p-2k$, hence $m$ is even.
Then $(q^{m/2} L)^\circledast\cong q^{-m/2} L^\circledast \cong
q^{m/2} L$ and we have proved the first part.
For the second part, suppose that $L \cong L^\circledast$ 
and that $P$ is the projective cover of $L$.
We need to show that the projective indecomposable module $P^\#$ also
covers $L$. This follows because
$\dim \Hom_H(P^\#, L)_0
=\dim\Hom_H(P, L^\circledast)_0 = 
\dim\Hom_H(P, L)_0 \neq 0$,
using (\ref{pairing}).
\end{proof}

Lemma~\ref{shift} implies that the sets
\begin{align}\label{cb}
\mathbf{B} &:= \left\{[P]\:\big|\:\text{for all $\#$-self-dual
    projective indecomposable $H$-modules
$P$}\right\},\\
\mathbf{B}^* &:= \left\{[L]\:\big|\:\text{for all
    $\circledast$-self-dual irreducible $H$-modules
$L$}\right\}\label{dcb}
\end{align}
give bases for $\Proj(H)$ and $\Rep(H)$, respectively, as free
$\Z[q,q^{-1}]$-modules.
Moreover these two bases are dual to each other with respect to the
pairing $(\cdot,\cdot)$.

\subsection*{Induction and restriction}
For $\beta,\gamma \in Q^+$, tensor product (horizontal composition) in the quiver Hecke category defines
a non-unital algebra
embedding $H_\beta \otimes H_\gamma \hookrightarrow H_{\beta+\gamma}$.
We denote the image of the identity
$1_\beta \otimes 1_\gamma \in H_\beta \otimes H_\gamma$
by $1_{\beta,\gamma} \in H_{\beta+\gamma}$. Then for
a graded left $H_{\beta+\gamma}$-module $U$ and
a graded left $H_\beta \otimes H_\gamma$-module $V$,
we set
\begin{equation*}
\Res^{\beta+\gamma}_{\beta,\gamma} U := 1_{\beta,\gamma} U,\qquad
\Ind^{\beta+\gamma}_{\beta,\gamma} V := H_{\beta+\gamma} 1_{\beta,\gamma} \otimes_{H_\beta
  \otimes H_\gamma} V,
\end{equation*}
which are naturally graded left $H_\beta \otimes H_\gamma$- and
$H_{\beta+\gamma}$-modules, respectively.
Both are exact functors; for induction this is a consequence of the
basis theorem (Theorem~\ref{qhbasis}).
The following Mackey-type theorem is one of the main reasons quiver
Hecke algebras are so amenable to purely algebraic techniques.

\begin{Theorem}[``Mackey filtration'']\label{mackey}
Suppose $\beta,\gamma,\beta',\gamma' \in Q^+$
are of heights $m,n,m',n'$, respectively, such that
$\beta+\gamma = \beta'+\gamma'$.
Setting $k := \min(m,n,m',n')$, let
$\{1 = w_0 < \cdots < w_k\}$ 
be the set of minimal length $S_{m'} \times S_{n'} \backslash
S_{m+n} / S_m \times S_n$-double coset representatives ordered
via the Bruhat order.
For any graded left $H_\beta \otimes H_\gamma$-module $V$,
there is a filtration
$$
0 = V_{-1}\subseteq V_0 \subseteq V_1 \subseteq\cdots\subseteq V_k = 
\Res^{\beta'+\gamma'}_{\beta',\gamma'} \circ \Ind^{\beta+\gamma}_{\beta,\gamma} (V)
$$
defined by $V_j := \sum_{i=0}^j\sum_{w \in (S_{m'} \times S_{n'}) w_i (S_m \times
  S_n)} 1_{\beta',\gamma'}\tau_w 1_{\beta,\gamma}\otimes V$.
Moreover there is a unique isomorphism of graded $H_{\beta'} \otimes
H_{\gamma'}$-modules
\begin{align*}
V_j / V_{j-1}
&\stackrel{\sim}{\rightarrow}
\bigoplus_{\beta_1,\beta_2,\gamma_1,\gamma_2}
 q^{-\beta_2 \cdot \gamma_1}\Ind^{\beta',\gamma'}_{\beta_1,
\gamma_1,\beta_2, \gamma_2} \circ I^* \circ
\Res^{\beta,\gamma}_{\beta_1,\beta_2,\gamma_1, \gamma_2} (V),\\
1_{\beta',
\gamma'} \tau_{w_j} 1_{\beta,\gamma}
\otimes 
v + V_{j-1}  &\mapsto \sum_{\beta_1,\beta_2,\gamma_1,\gamma_2}
1_{\beta_1,\gamma_1,\beta_2,\gamma_2} \otimes 1_{\beta_1,\beta_2,\gamma_1,\gamma_2} v,
\end{align*}
where 
$I:
H_{\beta_1} \otimes H_{\gamma_1}\otimes H_{\beta_2}
\otimes H_{\gamma_2}\stackrel{\sim}{\rightarrow}
H_{\beta_1} \otimes H_{\beta_2}\otimes H_{\gamma_1}
\otimes H_{\gamma_2}$ is the obvious isomorphism, and
the sums are taken over all
$\beta_1,\beta_2,\gamma_1,\gamma_2 \in Q^+$
such that 
$\beta_1+\beta_2 = \beta, \gamma_1+\gamma_2 = \gamma,
\beta_1+\gamma_1 = \beta', \beta_2+\gamma_2 = \gamma'$
and $\min(\operatorname{ht}(\beta_2), \operatorname{ht}(\gamma_1)) =
j$.
\end{Theorem}

\begin{proof}
Think about the picture
$$
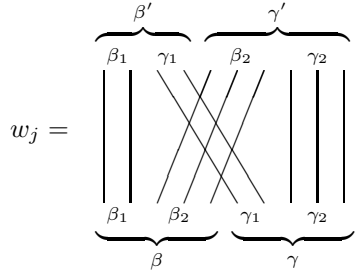
\begin{picture}(100,98)
\put(-25,45){$w_j=$}
\put(11,15){$_{\beta_1}$}
\put(34,15){$_{\beta_2}$}
\put(61,15){$_{\gamma_1}$}
\put(86,15){$_{\gamma_2}$}
\put(11,75){$_{\beta_1}$}
\put(30,75){$_{\gamma_1}$}
\put(57,75){$_{\beta_2}$}
\put(86,75){$_{\gamma_2}$}
\put(7,10){$\underbrace{\hspace{16mm}}_\beta$}
\put(58 ,10){$\underbrace{\hspace{16mm}}_\gamma$}
\put(7,80){$\overbrace{\hspace{12.5mm}}^{\beta'}$}
\put(48 ,80){$\overbrace{\hspace{19mm}}^{\gamma'}$}
\put(10,20){\line(0,1){50}}
\put(20,20){\line(0,1){50}}
\put(30,20){\line(2,5){20}}
\put(40,20){\line(2,5){20}}
\put(50,20){\line(2,5){20}}
\put(60,20){\line(-3,5){30}}
\put(70,20){\line(-3,5){30}}
\put(80,20){\line(0,1){50}}
\put(90,20){\line(0,1){50}}
\put(100,20){\line(0,1){50}}
\end{picture}
$$
\end{proof}

If $X$ is a graded left $H_\beta$-module and $Y$ is a graded left
$H_\gamma$-module, hence their outer tensor product $X \boxtimes Y$
is a graded left $H_\beta \otimes H_\gamma$-module, we set
$$
X \circ Y := \Ind_{\beta,\gamma}^{\beta+\gamma} (X \boxtimes Y).
$$
This is finite dimensional if both $X$ and $Y$ are finite dimensional
(by the basis theorem), and it is projective if both $X$ and $Y$
are projective (indeed we obviously have that $H 1_\bi \circ H 1_\bj =
H 1_{\bi\bj}$).
Hence $\circ$ defines a tensor product operation on both $\Rep(H)$ and
$\Proj(H)$,
making these into monoidal categories. So the Grothendieck groups
$[\Rep(H)]$ and $[\Proj(H)]$ become $\Z[q,q^{-1}]$-algebras.
Also introduce a bilinear multiplication $\circ$ on $\Z[q,q^{-1}]\W$, which we call the {\em
 shuffle product}, by declaring that
for words $\bi$ and $\bj$ of lengths $m$ and $n$, respectively, that
\begin{equation}
\bi \circ \bj := \sum_{\substack{w \in S_{m+n} \\ w(1) < \cdots <
    w(m)\\w(m+1) < \cdots < w(m+n)}}
q^{\deg(w;\bi\bj)} w(\bi\bj),
\end{equation}
recalling (\ref{degdef}).
This makes $\Z[q,q^{-1}]\W$ into an associative
$\Z[q,q^{-1}]$-algebra,
which we call the {\em quantum shuffle algebra}.
The first important consequence of the Mackey theorem (iterated!) is as follows.

\begin{Corollary}[``Shuffle Lemma'']\label{shuffle}
For finite dimensional $H$-modules $X$ and $Y$, 
we have that
$$
\Ch (X \circ Y) = (\Ch X) \circ (\Ch Y).
$$
Hence $\Ch:[\Rep(H)] \hookrightarrow \Z[q,q^{-1}]\W$ 
is an
injective algebra homomorphism.
\end{Corollary}

The restriction functor $\Res^{\beta+\gamma}_{\beta,\gamma}$
sends finite dimensional modules to finite dimensional modules
(obviously) and projectives to projectives (by the basis theorem).
Summing over all $\beta,\gamma$, we get induced
$\Z[q,q^{-1}]$-module homomorphisms
$[\Rep(H)] \rightarrow [\Rep(H)] \otimes_{\Z[q,q^{-1}]} [\Rep(H)]$ and
$[\Proj(H)] \rightarrow [\Proj(H)] \otimes_{\Z[q,q^{-1}]} [\Proj(H)]$
making our Grothendieck groups into coalgebras.
The Mackey theorem also implies the following.

\begin{Corollary}
The Grothendieck group $[\Rep(H)]$ is a twisted bialgebra (i.e.
comultiplication is an algebra homorphism)
with respect to the multiplication
on
$[\Rep(H)] \otimes_{\Z[q,q^{-1}]} [\Rep(H)]$ 
defined by
\begin{equation}\label{twist}
(a \otimes b) (c \otimes d) := q^{-\beta\cdot\gamma} ac
\otimes bd
\end{equation}
for $a \in [\Rep(H_\alpha)]$, $b \in [\Rep(H_\beta)]$, $c \in
[\Rep(H_\gamma)]$
and $d \in [\Rep(H_\delta)]$.
Similarly $[\Proj(H)]$ is a twisted bialgebra in the same sense.
\end{Corollary}

Note finally if we extend the pairing between $[\Proj(H)]$ and
$[\Rep(H)]$ diagonally to a pairing between $[\Proj(H)]\otimes_{\Z[q,q^{-1}]}
[\Proj(H)]$
and $[\Rep(H)]\otimes_{\Z[q,q^{-1}]} [\Rep(H)]$, i.e. we set
$(a\otimes b, c \otimes d) := (a,c) (b,d)$, then
the multiplication on $[\Rep(H)]$ is dual to the comultiplication on
$[\Proj(H)]$, and vice versa. This follows from the definition
(\ref{pairing}) together with Frobenius reciprocity.
In other words the $Q^+$-graded twisted bialgebras $[\Rep(H)]$ and $[\Proj(H)]$
are graded dual to each other.

\subsection*{The categorification theorem}
We now connect
$[\Rep(H)]$ and $[\Proj(H)]$ to Lusztig's algebra $\f$, that is, the $\Q(q)$-algebra on generators
$\theta_i\:(i \in I)$ 
subject to the {\em quantum Serre relations}
\begin{equation*}
\sum_{r+s=1 - \alpha_i \cdot \alpha_j}
(-1)^r
\theta_i^{(r)}\theta_j \theta_i^{(s)} = 0
\end{equation*}
for all $i, j \in I$ and $r \geq 1$,
where $\theta_i^{(r)}$ denotes the {\em divided power} $\theta_i^r /
[r]!$.
There is a $Q^+$-grading $\f = \bigoplus_{\alpha \in Q^+} \f_\alpha$
defined so that $\theta_i$ is in degree $\alpha_i$.
We use this to make $\f \otimes \f$ into an algebra with
twisted multiplication defined in the same way as (\ref{twist}).
Then there is a unique algebra homomorphism
\begin{equation}\label{twistedco}
r:\f \rightarrow \f \otimes \f,
\qquad
\theta_i\mapsto \theta_i \otimes 1 + 1 \otimes \theta_i
\end{equation}
making $\f$ into a twisted bialgebra.
We also need Lusztig's $\Z[q,q^{-1}]$-form $\f_{\Z[q,q^{-1}]}$ for
$\f$, which is the $\Z[q,q^{-1}]$-subalgebra of
$\f$ generated by all $\theta_i^{(r)}$.
We always identify $\f$ with $\Q(q) \otimes_{\Z[q,q^{-1}]} \f_{\Z[q,q^{-1}]}$.
The first important categorification theorem is as follows.

\begin{Theorem}[Khovanov-Lauda]\label{gamma}
There is an isomorphism of twisted bialgebras
$$
\gamma:\f_{\Z[q,q^{-1}]} \stackrel{\sim}{\rightarrow} [\Proj(H)]
$$
such that
$\theta_i \mapsto[H_{\alpha_i}]$ for each $i \in I$.
\end{Theorem}

In the rest of the subsection we are going to explain the proof of this,
since it is instructive, elementary, and reveals many further connections between
$\f$ and $H$.
To start with we extend scalars and show that there is a $\Q(q)$-algebra homomorphism
$$
\hat\gamma:\f \rightarrow \Q(q) \otimes_{\Z[q,q^{-1}]} [\Proj(H)],
\qquad
\theta_i \mapsto [H_{\alpha_i}].
$$
To prove this we must show that the Serre relations hold in the right hand algebra.
For $\bi = i_1\cdots i_n \in \W$ let $\theta_\bi := \theta_{i_1}\cdots \theta_{i_n}$.
Note that the homomorphsm $\gamma$ (if it exists) should send
$\theta_\bi$ to $[H 1_\bi]$. Hence,
recalling the definition (\ref{pin}) and using Corollary~\ref{ln},
it should send $\theta_i^{(n)}$ 
to $[P(i^n)]$.
Thus we need to show that
$$
\sum_{r+s=1-\alpha_i\cdot\alpha_j}(-1)^r [P(i^r) \circ P(j) \circ P(i^s)] = 0
$$
in $[\Proj(H)]$.
This follows from the following stronger statement, which gives a categorification of the Serre relations.

\begin{Lemma}\label{serrerelations}
Take $i, j \in I$ with $i \neq j$ and let $n := 1 - \alpha_i\cdot \alpha_j$.
Let $\alpha := n \alpha_i + \alpha_j$.
For $r,s\geq 0$ with $r+s=n$, let $e_{r,s} \in H_\alpha$ denote the image
of $e_r \otimes 1_j \otimes e_s$ under the embedding
$H_{r \alpha_i} \otimes H_{\alpha_j} \otimes H_{s \alpha_i} \hookrightarrow H_\alpha$,
so that $P(i^r) \circ P(j) \circ P(i^s)=
q^{\frac{1}{2}r(r-1)+\frac{1}{2}s(s-1)}H_\alpha e_{r,s}$.
Then there is an exact sequence
\begin{multline*}
0 \rightarrow q^{\frac{1}{2}n(n-1)}H_\alpha e_{0,n}
\rightarrow 
\cdots\\
\rightarrow q^{\frac{1}{2}r(r-1)+\frac{1}{2}s(s-1)} H_\alpha e_{r,s}
\stackrel{d_{r,s}}{\rightarrow} q^{\frac{1}{2}(r+1)r+\frac{1}{2}(s-1)(s-2)}
H_\alpha e_{r+1,s-1}\rightarrow\\\cdots\rightarrow
q^{\frac{1}{2}n(n-1)}H_\alpha e_{n, 0}
\rightarrow 0.
\end{multline*}
The homomorphism $d_{r,s}$ is given by right multiplication by
$$
\begin{picture}(100,98)
\put(-25,45){$\tau_{r,s}:=$}
\put(8,15){$_i$}
\put(18,15){$_i$}
\put(28,15){$_i$}
\put(38,15){$_i$}
\put(48,15){$_i$}
\put(58,15){$_j$}
\put(68,15){$_i$}
\put(78,15){$_i$}
\put(88,15){$_i$}
\put(98,15){$_i$}
\put(8,75){$_i$}
\put(18,75){$_i$}
\put(28,75){$_i$}
\put(38,75){$_i$}
\put(48,75){$_j$}
\put(58,75){$_i$}
\put(68,75){$_i$}
\put(78,75){$_i$}
\put(88,75){$_i$}
\put(98,75){$_i$}
\put(7,10){$\underbrace{\hspace{16mm}}_{{r+1}}$}
\put(68 ,10){$\underbrace{\hspace{12.5mm}}_{{s-1}}$}
\put(7,80){$\overbrace{\hspace{12.5mm}}^{r}$}
\put(58 ,80){$\overbrace{\hspace{16mm}}^{s}$}
\put(70,20){\line(0,1){50}}
\put(80,20){\line(0,1){50}}
\put(90,20){\line(0,1){50}}
\put(100,20){\line(0,1){50}}
\put(10,20){\line(1,1){50}}
\put(20,20){\line(-1,5){10}}
\put(30,20){\line(-1,5){10}}
\put(40,20){\line(-1,5){10}}
\put(50,20){\line(-1,5){10}}
\put(60,20){\line(-1,5){10}}
\end{picture}
$$
\end{Lemma}

\begin{proof}
It is a complex because $e_{r,s} \tau_{r,s} \tau_{r+1,s-1} = 0$. To see this
observe that the product
$\tau_{r,s}\tau_{r+1,s-1}$ can be rewritten as $\tau_{r+2}\pi$ for some $\pi$, and
$e_{r,s} \tau_{r+2} = 0$ by the definition (\ref{en}).
It is exact because it is homotopy equivalent to 0. To see this let
$
d_{r,s}': 
q^{\frac{1}{2}(r+1)r+\frac{1}{2}(s-1)(s-2)}
H_\alpha e_{r+1,s-1}
\rightarrow
q^{\frac{1}{2}r(r-1)+\frac{1}{2}s(s-1)} H_\alpha e_{r,s}
$
be defined by right multiplication by
$$
\begin{picture}(100,98)
\put(-73,45){$\tau_{r,s}':= (-1)^{r+m_{j,i}}$}
\put(8,15){$_i$}
\put(18,15){$_i$}
\put(28,15){$_i$}
\put(38,15){$_i$}
\put(48,15){$_j$}
\put(58,15){$_i$}
\put(68,15){$_i$}
\put(78,15){$_i$}
\put(88,15){$_i$}
\put(98,15){$_i$}
\put(8,75){$_i$}
\put(18,75){$_i$}
\put(28,75){$_i$}
\put(38,75){$_i$}
\put(48,75){$_i$}
\put(58,75){$_j$}
\put(68,75){$_i$}
\put(78,75){$_i$}
\put(88,75){$_i$}
\put(98,75){$_i$}
\put(7,10){$\underbrace{\hspace{12.5mm}}_{{r}}$}
\put(58 ,10){$\underbrace{\hspace{16mm}}_{{s}}$}
\put(7,80){$\overbrace{\hspace{16mm}}^{r+1}$}
\put(68 ,80){$\overbrace{\hspace{12.5mm}}^{s-1}$}
\put(10,20){\line(0,1){50}}
\put(20,20){\line(0,1){50}}
\put(30,20){\line(0,1){50}}
\put(40,20){\line(0,1){50}}
\put(100,20){\line(-1,1){50}}
\put(60,70){\line(-1,-5){10}}
\put(70,70){\line(-1,-5){10}}
\put(80,70){\line(-1,-5){10}}
\put(90,70){\line(-1,-5){10}}
\put(100,70){\line(-1,-5){10}}
\end{picture}
$$
Then check miraculously that $d_{r,s}' \circ d_{r,s} + d_{r-1,s+1}\circ d_{r-1,s+1}'= \operatorname{id}$.
\end{proof}

Hence the homomorphism $\hat\gamma$ is well-defined.
Moreover both $\f$ and $\Q(q)\otimes_{\Z[q,q^{-1}]} [\Proj(H)]$ are
twisted bialgebras in the same sense, $\hat\gamma$ respects the $Q^+$-gradings,
and it is easy to see that the analogue of (\ref{twistedco}) 
also holds in $[\Proj(H)]$. 
So $\hat\gamma$ is actually a homomorphism of twisted bialgebras.
We've already observed that it sends $\theta_i^{(r)}$
to $[P(i^n)] \in [\Proj(H)] \subset \Q(q) \otimes_{\Z[q,q^{-1}]} [\Proj(H)]$.
So $\hat \gamma$ restricts to a well-defined homomorphism
$\gamma$ as in the statement of Theorem~\ref{gamma}.

It remains to show that $\gamma$ is an isomorphism.
To establish its surjectivity, we dualize and exploit Corollary~\ref{chinjc}.
In more detail, Lusztig showed that $\f$ possesses a unique non-degenerate
symmetric bilinear form $(\cdot,\cdot)$ such that
$$
(1,1) = 1, 
\qquad
(\theta_i, \theta_j) = \frac{\delta_{i,j}}{1-q^{2}},
\qquad
(ab, c) = (a \otimes b, r(c))
$$
for all $i,j \in I$ and $a,b,c \in \f$; on the right hand side of the
last equation  $(\cdot,\cdot)$ is the product form $(a \otimes b, c \otimes d) :=
 (a,c) (b,d)$ 
on $\f \otimes \f$.
Let $\f_{\Z[q,q^{-1}]}^*$ be 
the dual of $\f_{\Z[q,q^{-1}]}$ with respect to the form $(\cdot,\cdot)$, i.e.
$$
\f_{\Z[q,q^{-1}]}^* := \left\{y \in \f\:\big|\:(x,y)\in\Z[q,q^{-1}]\text{ for all }x \in
  \f_{\Z[q,q^{-1}]}\right\}.
$$
It is 
another $\Z[q,q^{-1}]$-form for $\f$, that is, it is a twisted bialgebra such that
$\f = \Q(q)\otimes_{\Z[q,q^{-1}]} \f_{\Z[q,q^{-1}]}^*$.
Taking the dual map to $\gamma$ with respect to Lusztig's pairing between
$\f_{\Z[q,q^{-1}]}$ and $\f_{\Z[q,q^{-1}]}^*$ and the pairing
(\ref{pairing}) between $[\Proj(H)]$ and $[\Rep(H)]$ gives a twisted bialgebra homomorphism
$$
\gamma^*:[\Rep(H)] \rightarrow \f_{\Z[q,q^{-1}]}^*.
$$
For $x \in \f_{\Z[q,q^{-1}]}^*$, we let
$\chi(x) := \sum_{\bi\in\W} (\theta_\bi, x) \bi$,
thus defining an injective map
$\chi:\f_{\Z[q,q^{-1}]}^* \hookrightarrow \Z[q,q^{-1}]\W$.
Then we claim that the following diagram commutes:
\begin{equation*}
\begin{picture}(100,38)
\put(0,30){$[\Rep(H)] \stackrel{\gamma^*}{\longrightarrow}
\f_{\Z[q,q^{-1}]}^*$}
\put(16,14){$\searrow$}
\put(7,15){$_{\Ch}$}
\put(67,15){$_{\chi}$}
\put(59,14){$\swarrow$}
\put(21,0){$\Z[q,q^{-1}]\W$}
\end{picture}
\end{equation*}
To see this, just observe for a finite dimensional $H$-module $V$ that
\begin{equation}\label{inup}
(\theta_\bi, \gamma^*[V])
= (\gamma(\theta_\bi),[V])
= (H 1_\bi,[V])
= \Dim \Hom_H(H 1_\bi, V)
= \Dim 1_\bi V,
\end{equation}
which is indeed the coefficient of $1_\bi$ in $\Ch V$.
The map $\Ch$ is injective by Corollary~\ref{chinjc}, hence $\gamma^*$ is injective, and $\gamma$ is surjective as required.

Finally we must show that $\gamma$ is injective. 
To see this we lift the pairing between $[\Proj(H)]$ and $[\Rep(H)]$ to define a 
bilinear form on $[\Proj(H)]$ such that 
$$
([P],[Q]) := \Dim \Hom_H(P^\#, Q).
$$
At first sight this takes values in the ring of formal Laurent series
in $q$.
However, the surjectivity established in the previous paragraph implies that $[\Proj(H)]$ is spanned by the classes $\{[H 1_\bi]\:|\:\bi \in \W\}$, and 
$([H 1_\bj],[H 1_\bi]) = \Dim 1_\bj H 1_\bi$, which lies in $\Q(q)$ by (\ref{iden}).
Thus our form takes values in $\Q(q)$ and, extending scalars, we get induced a symmetric $\Q(q)$-bilinear form
$(\cdot,\cdot)$ on $\Q(q)\otimes_{\Z[q,q^{-1}]} [\Proj(H)]$.
It remains to observe from the definition of Lusztig's form above that
\begin{equation}\label{ip}
(\theta_\bj,\theta_\bi) = \frac{1}{(1-q^2)^n}\sum_{\substack{w \in S_n \\ w(\bi)=\bj}}
q^{\deg(w;\bi)}
\end{equation}
for $\bi,\bj \in \W$ with $\bi$ of length $n$.
This agrees with the right hand side of
(\ref{iden}), hence 
$(\hat\gamma(x),\hat\gamma(y)) = (x,y)$ for all $x,y \in \f$.
Now if $\hat\gamma(x) = 0$ for some $x \in \f$ we deduce from this 
that $(x,y) = 0$ for all $y \in \f$, hence $x = 0$ by the non-degeneracy of Lusztig's form. This shows that $\hat\gamma$, hence $\gamma$, is injective.
This completes the proof of Theorem~\ref{gamma}.

In the above argument, we have shown not only that
$\f_{\Z[q,q^{-1}]} \cong [\Proj(H)]$ but also that
$[\Rep(H)] \cong \f_{\Z[q,q^{-1}]}^*$ (both as twisted bialgebras).
We have also identified Lusztig's form with our representation-theoretically defined form (\ref{pairing}).
There is one other useful identification to be made at this point.
Let $\b:\f \rightarrow \f$
be the anti-linear algebra automorphism
such that $\b(\theta_i) = \theta_i$ for all $i \in I$.
Also let 
$\b^*:\f \rightarrow \f$
be the adjoint anti-linear
map to $\b$ 
with respect to Lusztig's form, so $\b^*$ is
defined from 
$(x, \b^*(y)) = \overline{(\b(x), y)}$
for any $x, y \in \f$.
The maps $\b$ and $\b^*$ preserve $\f_{\Z[q,q^{-1}]}$ and $\f_{\Z[q,q^{-1}]}^*$, respectively.
The map $\b^*$ is not an algebra homomorphism; instead it has the property
\begin{equation}
\b^*(xy) = q^{\beta\cdot \gamma} \b^*(y) \b^*(x)
\end{equation}
for $x$ of weight $\beta$ and $y$ of weight $\gamma$.
It is obvious for any $\bi\in\W$ that $\b(\theta_\bi) = \theta_\bi$
and $H 1_\bi^\# = H 1_\bi$, which is all that is needed to prove that
the isomorphism $\gamma$ intertwines $\b$
with the anti-linear involution on $[\Proj(H)]$ induced by the duality $\#$.
Because 
$\circledast$ is adjoint to $\#$ thanks to (\ref{pairing}),
we deduce that
$\gamma^*$ intertwines $\b^*$
with the anti-linear involution on $[\Rep(H)]$ induced by the duality $\circledast$.

For the remainder of the article we will simply {\em identify}
$[\Proj(H)]$ with $\f_{\Z[q,q^{-1}]}$
and $[\Rep(H)]$ with $\f_{\Z[q,q^{-1}]}^*$
via the maps $\gamma$ and $\gamma^*$. 
In particular the bases $\mathbf{B}$ and $\mathbf{B}^*$ from (\ref{cb})--(\ref{dcb}) give
bases for $\f_{\Z[q,q^{-1}]}$ and $\f_{\Z[q,q^{-1}]}^*$, respectively.
Here is the next remarkable result.

\begin{Theorem}[Rouquier, Varagnolo-Vasserot]\label{canb}
Assume $\K$ is of characteristic 0.
Then $\mathbf{B}$ coincides with the Lusztig-Kashiwara canonical basis for $\f$, and $\mathbf{B}^*$ is
the dual canonical basis.
\end{Theorem}

\subsection*{Some examples}
Here we give a few rather special examples which illustrate the algebraic techniques (characters, shuffle products, etc...) that we have developed so far.
First suppose that the graph underlying the quiver is the Dynkin diagram
$1$---$2$ of type
A$_2$, and consider $H_\alpha$ for the highest root
$\alpha = \alpha_1+\alpha_2$.
By Frobenius reciprocity, any irreducible graded $H_\alpha$-module
must appear in the head of one of $L(1) \circ L(2)$ or $L(2) \circ
L(1)$. By the Shuffle Lemma, we have that
$$
\Ch (L(1) \circ L(2)) = 12 + q 21,
\qquad
\Ch (L(2) \circ L(1)) = 21 + q 12.
$$
Because of Corollary~\ref{shift} this means that there can only be two
irreducible graded $H_\alpha$-modules up to isomorphism and degree shift,
namely, the one-dimensional modules
$L(12)$ and $L(21)$ with characters $12$ and $21$, respectively.
So in this case, $H_\alpha$ is already a {\em basic algebra}. In
fact it is easy to see directly that $H_\alpha$ is 
isomorphic to $A \otimes \K[x]$,
where $A$ is the path algebra of the quiver
\begin{equation}\label{A}
\begin{picture}(10,10)
\put(4,0.4){$\bullet$}
\put(10,2){$\longrightarrow$}
\put(10,-2){$\longleftarrow$}
\put(27,0.4){$\bullet$}
\put(15.5,9){$_\tau$}
\put(15.5,-4){$_\tau$}
\end{picture}
\end{equation}
graded by path length
and $x$ is of degree $2$ (corresponding to the central element $x_1
1_{12} + x_2 1_{21} \in H_\alpha$).
We deduce from this that $H_\alpha$ has global dimension 2.

We pause briefly to discuss {\em homogeneous representations}.
So return for a moment to a general quiver Hecke algebra.
Let $\sim$ be the equivalence relation on $\W$ 
generated by permuting adjacent pairs of letters $ij$ in a word
whenever $\alpha_i \cdot \alpha_j =0$. 
Call a word $\bi \in \W$ a {\em homogeneous word}
if 
every $\bj = j_1 \cdots j_n \sim \bi$ 
satisfies the following conditions:
\begin{itemize}
\item
$j_k \neq j_{k+1}$ for each
$k=1,\dots,n-1$;
\item if $j_k = j_{k+2}$
for some $k = 1,\dots,n-2$ then
$\alpha_{j_k}\cdot \alpha_{j_{k+1}} \neq -1$.
\end{itemize}
In that case, by the relations, there is a well-defined $H$-module
$L(\bi)$ concentrated in
degree zero 
with basis $\{v_\bj\:|\:\bj \sim \bi\}$ 
such that
\begin{itemize}
\item
$1_\bk v_\bj = \delta_{\bj,\bk} v_\bj$;
\item
$x_k v_\bj = 0$;
\item
$\tau_k v_\bj = v_{t_k(\bj)}$ if $\alpha_{j_k}\cdot \alpha_{j_{k+1}} =
0$, $\tau_k v_\bj  = 0$ otherwise.
\end{itemize}
It is obvious that $L(\bi)$ is irreducible.
Moreover for two homogeneous words $\bi,\bj$, we have 
that $L(\bi) \cong L(\bj)$ if and only if $\bi \sim \bj$.

\begin{Theorem}[Kleshchev-Ram]\label{homo}
Let $\homog$ be a set of representatives for the $\sim$-equivalence
classes
of homogeneous words $\bi \in \W$.
Up to isomorphism, the modules $\{L(\bi)\:|\:\bi \in \homog\}$
give all irreducible $H$-modules that are concentrated in degree zero.
\end{Theorem}

\begin{proof}
Let $L$ be an irreducible $H$-module that is concentrated in degree
zero.
Let $\bi$ be a word of $L$, i.e. a word such that $1_\bi L \neq 0$.
We claim first that $\bi$ is homogeneous. 
To see this, note for $\bj \sim \bi$
that $\dim 1_\bj L = \dim 1_\bi L$
since $\tau_k^2 1_\bi = 1_\bi$ whenever $\alpha_{i_k}\cdot\alpha_{i_{k+1}} =
0$. Hence $\bj$ is also a word of $L$.
If $j_k = j_{k+1}$ for some $k$ we get a contradiction to
the relation $(\tau_k x_{k+1} - x_{k} \tau_k) 1_\bj = 1_\bj$, since
the left hand side has to act on $L$ as zero by degree considerations.
Similarly if $j_k = j_{k+2}$ and $\alpha_{j_k}\cdot \alpha_{j_{k+1}} = -1$ for
some $k$
we get a contradiction to the relation 
$(\tau_{k+1} \tau_k \tau_{k+1} - \tau_k \tau_{k+1} \tau_k)1_\bj
= \pm 1_\bj$.
To complete the proof of the theorem, it
remains to observe that
$$
\Ch L = \sum_{\bi \in \homog} (\dim 1_\bi L) \Ch L(\bi)
$$
Hence $L \cong L(\bi)$ for some $\bi \in \homog$ by Corollary~\ref{chinjc}.
\end{proof}

Now back to examples. 
We next suppose that the graph underlying our quiver is the
A$_3$ Dynkin diagram
$1$---$2$---$3$, and 
take $\alpha = \alpha_1+\alpha_2+\alpha_3$. 
Because $\alpha$ is multiplicity-free, all words $\bi \in \W_\alpha$
are homogeneous, hence all irreducible graded
$H_\alpha$-modules are homogeneous representations.
Here are all of the skew-hooks with three boxes:
\begin{equation}
\begin{picture}(30,10)
\put(0,0){\line(1,0){30}}
\put(0,10){\line(1,0){30}}
\put(0,0){\line(0,1){10}}
\put(10,0){\line(0,1){10}}
\put(20,0){\line(0,1){10}}
\put(30,0){\line(0,1){10}}
\put(2.5,2){1}
\put(12.5,2){2}
\put(22.5,2){3}
\end{picture}
\qquad\:\:
\begin{picture}(20,20)
\put(0,-5){\line(1,0){10}}
\put(0,5){\line(1,0){20}}
\put(0,15){\line(1,0){20}}
\put(20,5){\line(0,1){10}}
\put(0,-5){\line(0,1){20}}
\put(10,-5){\line(0,1){20}}
\put(2.5,-3){1}
\put(2.5,7){2}
\put(12.5,7){3}
\end{picture}
\qquad\:\:
\begin{picture}(20,20)
\put(20,15){\line(-1,0){10}}
\put(20,5){\line(-1,0){20}}
\put(20,-5){\line(-1,0){20}}
\put(0,5){\line(0,-1){10}}
\put(20,15){\line(0,-1){20}}
\put(10,15){\line(0,-1){20}}
\put(2.5,-3){1}
\put(12.5,-3){2}
\put(12.5,7){3}
\end{picture}
\qquad\quad
\begin{picture}(10,20)
\put(0,-10){\line(0,1){30}}
\put(10,-10){\line(0,1){30}}
\put(0,-10){\line(1,0){10}}
\put(0,0){\line(1,0){10}}
\put(0,10){\line(1,0){10}}
\put(0,20){\line(1,0){10}}
\put(2.5,-8){1}
\put(2.5,2){2}
\put(2.5,12){3}
\end{picture}\label{skew}
\end{equation}

\vspace{2mm}

\noindent
We have filled in the boxes with their {\em contents} 1 2 3 in order
from southwest to northeast.
Reading contents along rows starting from the top row, we obtain
a distinguished set $\{123, 231, 312, 321\}$ of 
representatives for the $\sim$-equivalence classes of (homogeneous) words in
$\W_\alpha$.
The corresponding irreducible representations $L(123), L(231), L(312)$
and $L(321)$ have characters $123$, $231+213$, $312+132$ and $321$,
respectively.
They give all of the irreducible graded $H_\alpha$-modules up to
degree shift.
Their projective covers
$P(123), P(231), P(312)$ and $P(321)$ can be obtained explicitly
as the left ideals generated by the idempotents $1_{123}, 1_{231},
1_{312}$ and $1_{321}$, respectively.
This is so explicit that one can
then compute the endomorphism algebra of the resulting
minimal projective generator
$P(123)\oplus P(231)\oplus P(312)\oplus P(321)$
directly, to see that it is isomorphic to the tensor product $A \otimes A \otimes \K[x]$
where $A$ is the path algebra of the quiver (\ref{A}) as above.
This is a graded algebra of global dimension 3, and $H_\alpha$ is graded
Morita equivalent to it, so $H_\alpha$ has global dimension 3 too.

\begin{Theorem}[Brundan-Kleshchev]\label{unpub}
Suppose the graph underlying the quiver is the Dynkin diagram
$\operatorname{A}_n$ and that $\alpha = \alpha_1+\cdots+\alpha_n$ is the
highest root.
Then all irreducible graded $H_\alpha$-modules are homogeneous and are 
parametrized naturally by the skew-hooks with $n$ boxes as in the example above.
Moreover $H_\alpha$ is graded Morita equivalent
to $A^{\otimes (n-1)} \otimes \K[x]$, which is of global dimension $n$.
\end{Theorem}

In fact, as we'll discuss in more detail later on, all the quiver Hecke algebras whose underlying graph is a finite type Dynkin diagram
are of finite global dimension. We expect the converse of this statement holds too:
it should be the case that $H_\alpha$ has finite global dimension for all $\alpha \in Q^+$ if and only if the underlying graph is a finite type Dynkin diagram.

We end with one example of infinite global dimension: suppose the underlying graph is $0$=\!=$1$, that is, the affine Dynkin diagram A$_1^{(1)}$.
The words $01$ and $10$ are homogeneous, so we have the one-dimensional
homogeneous $H_\delta$-modules $L(01)$ and $L(10)$, where
$\delta = \alpha_0+\alpha_1$ is the smallest imaginary root.
Working in $\f = \Q(q) \otimes_{\Z[q,q^{-1}]} [\Rep(H)]$, we have that  
$\theta_i = [L(i)] / (1-q^2)$.
Hence
\begin{align*}
\theta_{01} &= [L(0)\circ L(1)] / (1-q^2)^2 =
([L(01)] + q^2 [L(10)]) / (1-q^2)^2,\\
\theta_{10} &= [L(1)\circ L(0)] / (1-q^2)^2 =
([L(10)] + q^2 [L(01)]) / (1-q^2)^2.
\end{align*}
We deduce that $[L(01)] =\frac{1-q^2}{1+q^2} (\theta_{01} - q^2 \theta_{10}).$
Using also (\ref{inup}), we can then compute the inner product
$$
([L(01)], [L(01)])=
\frac{1-q^2}{1+q^2} 
(\theta_{01} - q^2 \theta_{10}, [L(01)])
= \frac{1-q^2}{1+q^2} \notin \Z[q,q^{-1}].
$$
On the other hand, if $H_\delta$ has finite global dimension, we can find a finite projective resolution $P_n \rightarrow \cdots \rightarrow P_1 \rightarrow P_0 \rightarrow L(01) \rightarrow 0$,
to deduce that 
$$
([L(01)],[L(01)]) = \sum_{i=0}^n (-1)^i ([P_i],[L(01)]) \in \Z[q,q^{-1}].
$$
This contradiction establishes that $H_\delta$ has infinite global dimension.

\subsection*{Notes}
The main categorification theorem (Theorem~\ref{gamma}) is \cite[Theorem 1.1]{KL1},
and our exposition of the proof is based closely on the original account there.
The linear independence of characters (Theorem~\ref{chinj}) is \cite[Theorem 3.17]{KL1}; the proof given there is
essentially the same as the proof of the analogous result for degenerate affine Hecke algebras in Kleshchev's book \cite[Theorem 5.3.1]{Kbook} which in turn repeated an argument written down by Vazirani in her thesis based on the results of \cite{GV}; in the context of affine Hecke algebras 
this result goes back to Bernstein.
The second equality in (\ref{pairing}) is justified in \cite[Lemma 3.2]{KR2}.
The Mackey Theorem (Theorem~\ref{mackey}) is \cite[Proposition 2.18]{KL1}. 
The categorification of the Serre relations (Lemma~\ref{serrerelations})
was first worked out in a slightly weaker form in \cite[Corollary 7]{KL2}, and in the form described here in \cite[Lemma 3.13]{R1}.

The identification of the canonical and dual canonical bases (Theorem~\ref{canb}) is \cite[Corollary 5.8]{R2} or \cite[Theorem 4.5]{VV}.
The proof of this theorem depends on the geometric realization of quiver Hecke algebras, hence is valid only for the case of symmetric Cartan matrices.
If $\K$ is of positive characteristic then the bases 
$\mathbf{B}$ and
$\mathbf{B}^*$ arising from the quiver Hecke algebras 
are different in general from the canonical and dual canonical bases.
For a while there was a conjecture formulated by Kleshchev and Ram \cite[Conjecture 7.3]{KR2} asserting that they should be the same independent of characteristic at least in all finite ADE types, but this turned out to be false. Various counterexamples are explained by Williamson in \cite{W}; see also \cite[Example 2.16]{BKM}.
Note though that $\mathbf{B}^*$ is always a 
{\em perfect basis} in the sense of Berenstein-Kazhdan \cite{BK}.

The classification of homogeneous representations (Theorem~\ref{homo}) was worked out by Kleshchev and Ram in \cite[Theorem 3.4]{KR1}. The version proved here is a slight generalization since we include quivers with multiple edges.
In \cite[Theorem 3.10]{KR1}, 
Kleshchev and Ram go on to introduce a special class of
``minuscule'' homogeneous representations, showing that the dimensions of these representations are given by the Petersen-Proctor hook formula, 
thus revealing another unexpected connection to combinatorics.

Theorem~\ref{unpub} is unpublished but makes a good exercise!

\section{Finite type}

In the final section we restrict our attention to finite ADE types, i.e. we assume that the graph underlying the quiver is a finite ADE Dynkin diagram.

\subsection*{PBW and dual PBW bases}
Since we are now in finite ADE type, the underlying Kac-Moody algebra $\mathfrak{g}$ is actually a finite dimensional simple Lie algebra.
Let $W$ denote the (finite) Weyl group associated to $\mathfrak{g}$, which is the subgroup of $GL(P)$ generated by the {\em simple reflections} 
$\{s_i\:|\:i \in I\}$
defined from $s_i(\lambda) = \lambda - (\alpha_i \cdot \lambda) \alpha_i$.
Let $R := \bigcup_{i \in I} W(\alpha_i) \subset Q$ be the set of {\em roots}
and $R^+ := R \cap \bigoplus_{i \in I} \N \alpha_i$
be the {\em positive roots}.
We fix once and for all a reduced expression
$w_0 = s_{i_1} \cdots s_{i_N}$
for the longest word $w_0 \in W$ (so $N = |R^+|$).
There is a corresponding {\em convex order}
$\prec$ on $R^+$ defined from
$\alpha_{i_1} \prec s_{i_1}(\alpha_{i_2}) \prec\cdots \prec
s_{i_1}\cdots s_{i_{N-1}}(\alpha_{i_N})$.
By a {\em Kostant partition} of $\alpha \in Q^+$, we mean a sequence
$\lambda = (\lambda_1,\dots,\lambda_l)$ of positive roots
summing to $\alpha$
such that $\lambda_1 \succeq \cdots \succeq \lambda_l$ with respect to this fixed convex order.
Let $\KP(\alpha)$ denote the set of all Kostant partitions of $\alpha$
and $\KP := \bigcup_{\alpha \in Q^+} \KP(\alpha)$.

Associated to the reduced expression/convex order just fixed,
Lusztig has introduced a 
{\em PBW basis} for $\f$, indexed as it should be by $\KP$.
To construct this, Lusztig first defines
{\em root vectors} $\{r_\alpha\:|\:\alpha \in R^+\}$.
Then the {\em PBW monomial} associated to
$\lambda = (\lambda_1,\dots,\lambda_l) \in \KP$ is defined from
\begin{equation}\label{tu}
r_\lambda := r_{\lambda_1} \cdots r_{\lambda_l} / [\lambda]!,
\end{equation}
where $[\lambda]! := \prod_{\beta \in R^+} m_\beta(\lambda)$, and
$m_\beta(\lambda)$ denotes the multiplicity of $\beta$ in $\lambda$.
In Lusztig's approach, the definition of the root vector $r_\alpha$ 
depends on a certain action of the braid group $\langle T_i\:|\:i \in I\rangle$ associated to $W$ 
on the full quantized enveloping algebra $U_q(\mathfrak{g})$: fixing an embedding
$\f \hookrightarrow U_q(\mathfrak{g})$ one sets
$r_\alpha := T_{i_1} \cdots T_{i_{r-1}}(\theta_{i_r})$
if $\alpha = s_{i_1}\cdots s_{i_{r-1}}(\alpha_{i_r})$.
We skip the precise details here because there is also a more elementary recursive approach to the definition of the root vector $r_\alpha$,
well known in type A but only recently established in full generality in types D and E. 
To formulate this we need the notion of a {\em minimal pair}
for $\alpha \in R^+$: a pair $(\beta,\gamma)$ of positive roots 
such that $\beta\succ \gamma$,
$\beta+\gamma=\alpha$, and
there is no other pair $(\beta',\gamma')$ of positive roots with
$\beta'+\gamma' = \alpha$ and $\beta \succ \beta' \succ \alpha \succ \gamma' \succ \gamma$.

\begin{Lemma}\label{ral}
Let $\alpha \in R^+$. If $\alpha = \alpha_i$ for some $i \in I$ then $r_\alpha = \theta_i$.
If $\alpha$ is not simple then
$r_\alpha = r_\gamma r_\beta - q r_\beta r_\gamma$
for any minimal pair $(\beta,\gamma)$ for $\alpha$.
\end{Lemma}

\begin{proof}
This follows from Theorem~\ref{Ses} below.
\end{proof}

Also define the {\em dual root vectors} $\{r_\alpha^*\:|\:\alpha \in R^+\}$
by setting $r_\alpha^* := (1-q^2)r_\alpha$. Then 
for $\lambda = (\lambda_l,\dots,\lambda_l) \in \KP$ define
the {\em dual PBW monomial}
\begin{equation}\label{rlambdadual}
r_\lambda^* := q^{s_\lambda} r_{\lambda_1}^* \cdots r_{\lambda_l}^*
\end{equation}
where $s_\lambda := \frac{1}{2} \sum_{\beta \in R^+} m_\beta(\lambda) (m_\beta(\lambda)-1)$.
Lusztig's fundamental result about PBW and dual PBW bases is as follows.

\begin{Theorem}[Lusztig]
The sets $\{r_\lambda\:|\:\lambda \in \KP\}$ and
$\{r_\lambda^*\:|\:\lambda \in \KP\}$ give bases
for $\f_{\Z[q,q^{-1}]}$ and $\f_{\Z[q,q^{-1}]}^*$, respectively.
Moreover these two bases are dual with respect to the pairing $(\cdot,\cdot)$.
\end{Theorem}

For a Kostant partition $\lambda = (\lambda_1,\dots,\lambda_l)$,
we set
$\lambda_k' := \lambda_{l+1-k}$ for $k=1,\dots,l$.
Then
introduce a partial order $\preceq$ on $\KP$ by declaring that 
$\lambda \prec \mu$ if and only if both of the following hold:
\begin{itemize}
\item 
$\lambda_1 = \mu_1,\dots,\lambda_{k-1} = \mu_{k-1}$ and
$\lambda_k \prec \mu_k$ for some $k$ such that $\lambda_k$ and $\mu_k$
both make sense;
\item
$\lambda_1' = \mu_1',\dots,\lambda_{k-1}' = \mu_{k-1}'$ and
$\lambda_k'\succ \mu_k'$ for some $k$ such that $\lambda_k'$ and $\mu_k'$
both make sense.
\end{itemize}
We note for $\alpha \in R^+$ that the unique smallest element of $\KP(\alpha)$
is the one-part Kostant partition $(\alpha)$,
and the next smallest elements 
are the minimal pairs $(\beta,\gamma)$ defined above.
The important point is that the bar involutions $\b$ and $\b^*$ act on the PBW and dual PBW bases in a triangular way
with respect to this order:
\begin{align}\label{tr1}
\b(r_\lambda) &= r_\lambda + (\text{a $\Z[q,q^{-1}]$-linear
  combination of $r_\mu$ for $\mu \succ \lambda$}),\\
\b^*(r_\lambda^*) &= r_\lambda^* + (\text{a $\Z[q,q^{-1}]$-linear
  combination of $r_\mu^*$ for $\mu \prec \lambda$}).\label{tr2}
\end{align}
Combined with Lusztig's lemma, this triangularity implies the existence of unique bases
$\{b_\lambda\:|\:\lambda \in \KP\}$ and $\{b_\lambda^*\:|\:\lambda \in
\KP\}$
for $\f_{\Z[q,q^{-1}]}$ and $\f_{\Z[q,q^{-1}]}^*$, respectively, such that
\begin{align}\label{l}
\b(b_\lambda) &= b_\lambda,
 &b_\lambda &= r_\lambda + \text{(a $q \Z[q]$-linear combination of
  $r_\mu$ for $\mu \succ \lambda$)},\\
\b^*(b^*_\lambda) &= b^*_\lambda,
&b_\lambda^*  &= r^*_\lambda
 + \text{(a $q \Z[q]$-linear combination of
  $r^*_\mu$ for $\mu \prec \lambda$).}\label{m}
\end{align}
The basis $\{b_\lambda\:|\:\lambda \in \KP\}$ is 
the Lusztig-Kashiwara canonical basis.
The following lemma shows that $\{b_\lambda^*\:|\:\lambda \in \KP\}$ is the dual canonical
basis.

\begin{Lemma}
For $\lambda,\mu \in \KP$ we have that $(b_\lambda, b_\mu^*) = \delta_{\lambda,\mu}$.
\end{Lemma}

\begin{proof}
By (\ref{l})--(\ref{m}) and the duality of the PBW and dual PBW basis
vectors, we have that $(b_\lambda, b_\mu^*) \in \delta_{\lambda,\mu} +
q \Z[q]$. 
Now it remains to observe that it is bar-invariant:
$$
(b_\lambda,b_\mu^*)
= (\b(b_\lambda), b_\mu^*)
= \overline{(b_\lambda, \b^*(b_\mu^*))}
= \overline{(b_\lambda,b_\mu^*)}.
$$
\end{proof}

\subsection*{Proper standard modules}
If $\lambda = (\alpha)$ for $\alpha \in R^+$ then it is
minimal in $\KP(\alpha)$, hence by (\ref{m})
the dual canonical basis element $b_\lambda^*$ must simply be equal to
the dual root element $r_\alpha^*$.
For $\K$ of characteristic 0, we deduce immediately 
from this and the geometric Theorem~\ref{canb} 
that for each $\alpha \in R^+$ there is a 
(unique up to isomorphism) irreducible $H$-module $L(\alpha)$ such that $[L(\alpha)] = r_\alpha^*$.
These modules are known as {\em cuspidal modules}. 
McNamara has recently found
an elementary inductive proof of the existence of cuspidal modules
based just on Theorem~\ref{gamma},
which works for ground fields $\K$ of positive characteristic too. 
In fact in many cases, including all positive roots and all convex
orders in type A,  
the cuspidal module $L(\alpha)$ turns out to be
homogeneous, so can be constructed explicitly via
Theorem~\ref{homo}.

Now suppose we are given $\alpha \in Q^+$ and $\lambda = (\lambda_1,\dots,\lambda_l)\in \KP(\alpha)$.
Define
the {\em proper standard module}
\begin{equation}\label{psm}
\bar\Delta(\lambda) :=
q^{s_\lambda}
L(\lambda_1)\circ\cdots\circ L(\lambda_l).
\end{equation}
Comparing with the definition (\ref{rlambdadual}), we have that
that
$[\bar\Delta(\lambda)] = r_\lambda^*$,
i.e. proper standard modules categorify the dual PBW basis.
Then let
\begin{equation*}
L(\lambda) := \bar\Delta(\lambda) / \operatorname{rad} \bar\Delta(\lambda).
\end{equation*}
Now we can state the following classification of irreducible $H$-modules.

\begin{Theorem}[Kleshchev-Ram, Kato, McNamara]\label{mac2}
The modules $\{L(\lambda)\:|\:\lambda \in \KP\}$
give a complete set of pairwise inequivalent $\circledast$-self-dual
irreducible $H$-modules.
Moreover, all composition factors
of $\operatorname{rad} \bar\Delta(\lambda)$ 
are of the form $q^d L(\mu)$ for $\mu \prec \lambda$ and
$d \in \Z$. 
\end{Theorem}

Theorem~\ref{mac2} can be viewed as a vast generalization of Zelevinsky's
classification of irreducible representations of affine Hecke algebras
via ``multisegments.''
Zelevinsky's result can be interpreted as treating the case that the Dynkin diagram is
$1$---$2$---$\cdots$---$n$,
so the positive roots are
$\alpha_{i,j} := \alpha_i+\alpha_{i+1}+\cdots+\alpha_{j-1}$
for $1 \leq i \leq j \leq n$, and 
the convex order is
defined by $\alpha_{i,j} \prec \alpha_{k,l}$
if and only if $i < k$ or $i = k$ and $j < l$.
The cuspidal module $L(\alpha_{i,j})$ is the one-dimensional
homogeneous module corresponding to the homogeneous word
$[i,j] := i (i+1)\cdots j$. We call this a 
{\em segment}.
Then a Kostant partition $\lambda =
(\alpha_{i_1,j_1},\dots,\alpha_{i_r,j_r})$ is the same thing as a 
{\em multisegment}, i.e. a non-increasing sum of segments
$[i_1,j_1] +\cdots + [i_r,j_r]$. The
proper standard module
$\bar\Delta(\lambda)$ has character 
$q^{s_\lambda} [i_1,j_1]\circ \cdots [i_r,j_r]$ obtained by taking the
shuffle product of these segments. The irreducible heads of these
modules give all the irreducible $H$-modules up to isomorphism and
degree shift.
For example if $\alpha = \alpha_{1,3}$ then the irreducible graded
$H_\alpha$-modules
are indexed by
$\KP(\alpha) = \{(\alpha_{1,3}), (\alpha_{2,3},\alpha_1),
(\alpha_3,\alpha_{1,2}), (\alpha_3,\alpha_2,\alpha_1)\}$.
These modules are the homogeneous representations
parametrized by the skew-hooks from (\ref{skew}).
To translate from 
skew-hook to Kostant partition, 
record the contents in the rows of the skew-hook
from top to bottom to obtain the corresponding multisegment.

If $\K$ is of characteristic zero then we have that
$[\bar\Delta(\lambda)] = r_\lambda^*$ and $[L(\mu)] = b_\mu^*$
thanks to Theorem~\ref{canb}. Thus the coefficients 
$p_{\lambda,\mu}(q)$ defined from the equivalent expansions
$$
r_\lambda^* = \sum_{\mu \in \KP} p_{\lambda,\mu}(q) b_\mu^*,
\qquad
b_\mu = \sum_{\lambda \in \KP} p_{\lambda,\mu}(q) r_\lambda
$$
compute the graded
composition multiplicities $[\bar\Delta(\lambda):L(\mu)]$.
Thus Theorem~\ref{mac2} has the following application.

\begin{Corollary}\label{ka}
We have that 
$p_{\lambda,\mu}(q) = 0$ unless $\mu \preceq \lambda$,
$p_{\lambda,\lambda}(q) = 1$, 
and $p_{\lambda,\mu}(q) \in  q \N[q]$ if $\mu \prec \lambda$.
\end{Corollary}

Finally we record a lemma which is
at the heart of McNamara's inductive proof of the existence of
cuspidal modules and his approach to Theorem~\ref{mac2}. 
We will apply this several times also in the next subsection.

\begin{Lemma}[``McNamara's Lemma'']\label{maclem}
Suppose we are given $\alpha \in R^+$
and $\beta,\gamma \in Q^+$ with $\beta+\gamma=\alpha$.
If $\Res^{\alpha}_{\beta,\gamma} L(\alpha) \neq 0$
then $\beta$ is a sum of positive roots $\preceq \alpha$ and $\gamma$ is a sum
of positive roots $\succeq \alpha$.
\end{Lemma}

\begin{proof}
Forget about the grading throughout this proof.
We just explain how to show that $\beta$ is a sum of positive roots
$\preceq \alpha$; the statement about $\gamma$ follows similarly.
In the next paragraph, we'll prove the following claim:

\vspace{1mm}
\noindent{\bf Claim.}
{\em If $\Res^{\alpha}_{\beta,\gamma} L(\alpha) \neq 0$
for $\beta \in R^+$, $\gamma \in Q^+$ with 
$\beta+\gamma = \alpha$ then $\beta \preceq \alpha$.}

\vspace{1mm}

\noindent
Now suppose that we are in the situation of the lemma for non-zero
$\beta,\gamma$.
Let $L(\lambda) \boxtimes L(\mu)$ be a composition factor of $\Res^{\alpha}_{\beta,\gamma} L(\alpha)$ for some $\lambda \in \KP(\beta)$
and
$\mu \in \KP(\gamma)$.
Then it's clear that $\Res^{\alpha}_{\lambda_1, \alpha - \lambda_1}
L(\alpha)$ is non-zero.
Invoking the claim, we deduce that $\lambda_1 \preceq \alpha$.
Hence the parts of $\lambda$ are positive roots $\preceq \alpha$ 
summing to $\beta$, and we are done.

To prove the claim,
suppose for a contradiction that it is false, and look at the
counterexample in which $\beta$ is maximal; of course we have that
$\beta \succ \alpha$ and $\gamma \neq 0$.
Let $L(\lambda) \boxtimes L(\mu)$ be a submodule
of 
$\Res^{\alpha}_{\beta,\gamma} L(\alpha)$
for $\lambda \in \KP(\beta), \mu \in \KP(\gamma)$.
As in the previous paragraph, we have that 
$\Res^{\alpha}_{\lambda_1, \alpha-\lambda_1}
L(\alpha) \neq 0$, so by the maximality of $\beta$
we must have that $\lambda_1 \preceq \beta$.
Since $(\beta)$ is minimal in $\KP(\beta)$ this implies that
$\lambda = (\beta)$.
We have shown that $L(\beta) \boxtimes L(\mu)$ is a submodule of 
$\Res^{\alpha}_{\beta,\gamma} L(\alpha)$
for some $\mu \in \KP(\gamma)$.
By Frobenius reciprocity, we get a non-zero homomorphism
$L(\beta) \circ L(\mu) \rightarrow L(\alpha)$.
If $\beta \succeq \mu_1$ then $L(\beta) \circ L(\mu)$ is a
quotient of $\bar\Delta(\beta\sqcup \mu)$, where we write
$\beta\sqcup\mu$ for the Kostant partition obtained from $\mu$ by
adding the root $\beta$ to the beginning.
By Theorem~\ref{mac2} we know that 
$\bar\Delta(\beta\sqcup\mu)$ has irreducible head $L(\beta \sqcup
\mu)$, and we deduce that $L(\beta \sqcup \mu)\cong L(\alpha)$. This shows that
$\beta \sqcup \mu = (\alpha)$, hence $\beta = \alpha$,
contradicting $\beta \succ \alpha$.

So we have that $\beta \prec \mu_1$. Let  
$\bar\mu \in \KP(\gamma-\mu_1)$
be the (possibly empty) Kostant partition obtained from $\mu$ by removing its first
part, 
so $\mu = \mu_1 \sqcup \bar\mu$.
Since $L(\mu)$ is a quotient of $L(\mu_1) \circ \bar\Delta(\bar\mu)$,
Frobenius reciprocity gives us a non-zero homomorphism
$(L(\beta) \circ L(\mu_1)) \boxtimes \bar\Delta(\bar\mu) \rightarrow
\Res^{\alpha}_{\beta+\mu_1, \gamma-\mu_1} L(\alpha)$.
Hence there's some  $\nu \in \KP(\beta+\mu_1)$ such that
$L(\nu) \boxtimes L(\bar\mu)$ is a composition factor of
$\Res^{\alpha}_{\beta+\mu_1, \gamma-\mu_1} L(\alpha)$.
As above, the maximality of $\beta$ implies that $\nu_1 \preceq
\beta$, hence all parts of $\nu$ are $\preceq \beta$.
The parts of $\nu$ sum to $\beta+\mu_1$, so we deduce that
$\beta+\mu_1$ is a sum of roots $\preceq \beta$, as well as already
being a sum of $\beta$ and another root $\mu_1 \succ \beta$.
This contradicts a general property of convex orders (e.g. see
\cite[Lemma 2.4]{BKM}).
\end{proof}

\subsection*{Standard modules and homological algebra}
Just as the proper standard modules $\bar\Delta(\lambda)$ categorify
the dual PBW basis vectors $r_\lambda^*$, there are {\em standard
  modules}
$\Delta(\lambda)$ which categorify the PBW basis vectors $r_\lambda$.
The definition of these parallels the definition of the PBW basis elements:
we first define {\em root modules} $\Delta(\alpha)$ for each
$\alpha \in R^+$ so that $[\Delta(\alpha)] = r_\alpha$, 
then introduce {\em divided power modules}
$\Delta(\alpha^m)$ for $\alpha \in R^+$ and $m \geq 1$
so that $[\Delta(\alpha^m)] = r_\alpha^m / [m]!$, then finally
set
\begin{equation}\label{D}
\Delta(\lambda) := \Delta(\mu_1^{m_1}) \circ\cdots \circ
\Delta(\mu_k^{m_k})
\end{equation}
for a Kostant partition $\lambda$ written as
$\lambda = (\mu_1^{m_1},\dots,\mu_k^{m_k})$
for $\mu_1 \succ \cdots \succ \mu_k$.
It is then automatic from (\ref{tu}) that
$[\Delta(\lambda)] = r_\lambda$.

Perhaps the easiest way to define the root module $\Delta(\alpha)$ is
by inducing the cuspidal module $L(\alpha)$ from the previous
subsection (although this only works in simply-laced types).
Given $\alpha \in R^+$ of height $n$, let $H_\alpha'$ be the subalgebra of
$H_\alpha$ generated by the elements $\{1_\bi, \tau_j, x_j - x_{j+1}\:|\:\bi \in
\W_\alpha, 1 \leq j \leq n-1\}$. Then set
\begin{equation}\label{dal}
\Delta(\alpha) := H_\alpha \otimes_{H_{\alpha}'} L(\alpha).
\end{equation}
Recalling (\ref{centraldef}), we can choose the special word $\bi$ there so that 
$i_1$ appears with a multiplicity $k$ which is non-zero in the field
$\K$;  this is always possible by inspection of the ADE root systems.
Then, letting $z := z_1+\cdots+z_k \in Z(H_\alpha)_2$, we have that
$H_\alpha = H_\alpha' \otimes \K[z]$ as an algebra.
Hence we have that 
$\Delta(\alpha) = L(\alpha) \boxtimes \K[z]$ as a module, i.e. it is
an infinite self-extension of copies of $L(\alpha)$.
The following is obvious from this description combined with
Schur's Lemma.

\begin{Lemma}\label{rootmod}
The root module $\Delta(\alpha)$ has irreducible head $L(\alpha)$,
and we have that $[\Delta(\alpha)] = [L(\alpha)] / (1-q^2)$ in the Grothendieck
group, hence $[\Delta(\alpha)] = r_\alpha$.
Moreover $\End_{H_\alpha}(\Delta(\alpha)) = \K[z]$.
\end{Lemma}

There is also a quite different recursive description of
$\Delta(\alpha)$ which is essentially a
categorification of Lemma~\ref{ral}.
First if $\alpha = \alpha_i$ for $i \in I$ then we have that
$\Delta(\alpha)  \cong H_{\alpha_i}$.
Then if $\alpha \in R^+$ is of height at least two, we pick a minimal
pair
$(\beta,\gamma)$ for $\alpha$. As $[\Delta(\beta)] = (1+q^2+q^4+\cdots)[L(\beta)]$,
all composition factors of $\Delta(\beta)$ are isomorphic to
$L(\beta)$ up to a degree shift.
Similarly all composition factors of $\Delta(\gamma)$ are isomorphic
to  $L(\gamma)$ up to a degree shift.
Now an application of McNamara's Lemma reveals that there is only one non-zero layer in the
Mackey
filtration of $\Res^{\alpha}_{\beta,\gamma} \Delta(\gamma)\circ \Delta(\beta)$
from Theorem~\ref{mackey}, and this layer is isomorphic to
$q \Delta(\beta) \boxtimes \Delta(\gamma)$.
Hence $\Res^\alpha_{\beta,\gamma}
\Delta(\gamma) \circ \Delta(\beta)\cong q \Delta(\beta) \boxtimes
\Delta(\gamma)$.
Frobenius reciprocity then gives
us a canonical homomorphism
$$
\phi:q \Delta(\beta)\circ \Delta(\gamma)
\rightarrow \Delta(\gamma)\circ \Delta(\beta),
\qquad 1_\alpha \otimes (v_1 \otimes v_2) 
\mapsto \tau_w \otimes (v_2 \otimes v_1)
$$
where $w$ is the permutation $(1,\dots,m+n) \mapsto
(n+1,\dots,m+n,1,\dots,n)$
for $m := \height(\beta)$ and $n := \height(\gamma)$.

\begin{Theorem}[Brundan-Kleshchev-McNamara]\label{Ses}
The homomorphism $\phi$ just defined is injective
and its cokernel is isomorphic to the root module $\Delta(\alpha)$
from (\ref{dal}). Hence there is a short exact sequence
$$
0 \longrightarrow q \Delta(\beta) \circ \Delta(\gamma)
\stackrel{\phi}{\longrightarrow} 
\Delta(\gamma)\circ \Delta(\beta)
\longrightarrow \Delta(\alpha) \longrightarrow 0.
$$
\end{Theorem}

The proof of Theorem~\ref{Ses} depends in a crucial way on some 
homological algebra: it requires knowing in advance that
\begin{equation}\label{ext1}
\operatorname{Ext}^1_{H_\alpha}(L(\alpha), L(\alpha)) \cong q^{-2}\K,
\qquad
\operatorname{Ext}^i_{H_\alpha}(L(\alpha), L(\alpha)) = 0\text{ for $i
  \geq 2$.}
\end{equation}
This is surprisingly tricky to prove. McNamara found a proof assuming the finiteness of global dimension of $H_\alpha$,
and another way
based on some explicit computations which can be made just for certain special choices of the convex order (for which almost all $L(\alpha)$ are homogeneous).
He also showed how to deduce the finiteness of global dimension from (\ref{ext1}).
Thus one first proves (\ref{ext1}) for a particular convex order, then deduces the finiteness of global dimension, then from that extends (\ref{ext1}) to arbitrary convex orders, and finally uses that to prove 
Theorem~\ref{Ses}.

Next we explain how the divided powers $\Delta(\alpha^m)$ are defined.
Using the decomposition $\Delta(\alpha) = L(\alpha) \boxtimes \K[z]$,
it is easy to deduce from (\ref{ext1}) that
\begin{equation}\label{ext2}
\operatorname{Ext}^i_{H_\alpha}(\Delta(\alpha),\Delta(\alpha)) = 
\operatorname{Ext}^i_{H_\alpha}(\Delta(\alpha),L(\alpha)) = 0
\text{ for $i \geq 1$.}
\end{equation}
Using McNamara's Lemma once again, one checks that the Mackey filtration of
$\Res^{2\alpha}_{\alpha,\alpha} \Delta(\alpha) \circ \Delta(\alpha)$
has exactly two non-zero layers, $\Delta(\alpha) \boxtimes \Delta(\alpha)$
at the bottom and $q^{-2} \Delta(\alpha)\boxtimes \Delta(\alpha)$ at the top.
Then (\ref{ext2}) implies that this extension splits, hence there is a 
non-zero homogeneous homomorphism $\Delta(\alpha) \boxtimes \Delta(\alpha)
\rightarrow \Res^{2\alpha}_{\alpha,\alpha} \Delta(\alpha) \circ \Delta(\alpha)$
of degree $-2$. Applying Frobenius reciprocity we get from this a canonical 
degree $-2$ endomorphism
$$
\tau:\Delta(\alpha) \circ \Delta(\alpha) \rightarrow \Delta(\alpha) \circ \Delta(\alpha), \qquad 1_{2\alpha} \otimes (v_1 \otimes v_2)
\mapsto \tau_w \otimes (v_2 \otimes v_1),
$$
where $w:(1,\dots,2n) \mapsto (n+1,\dots,2n,1,\dots,n)$.
Recalling Lemma~\ref{rootmod}, one then shows that 
$\Delta(\alpha)$ has a {\em unique} degree 2 endomorphism $x$ such that
 $\tau \circ (1 x) = (x 1) \circ \tau + 1$ on $\Delta(\alpha) \circ \Delta(\alpha)$.
Then setting $\tau_i:= 1^{i-1} \tau 1^{m-i-1}$ and 
$x_i := 1^{i-1} x 1^{m-i}$ we obtain a right action of the nil Hecke algebra $\NH_m$ on $\Delta(\alpha)^{\circ m}$. 
Finally, recalling the idempotent $e_m \in \NH_m$ from (\ref{en}), we set
\begin{equation}
\Delta(\alpha^m) := q^{\frac{1}{2}m(m-1)}\Delta(\alpha)^{\circ m} e_m,
\end{equation}
and get from Corollary~\ref{ln}
that $\Delta(\alpha)^{\circ m} = [m]! \Delta(\alpha^m)$.
Hence we have indeed categorified the divided power $r_\alpha^m / [m]!$

This completes our sketch of the definition of the standard modules
$\Delta(\lambda)$.
The following theorem collects various
homological properties of the finite type quiver Hecke algebras.
Most of these are deduced from the definition (\ref{D}), either
using generalized Frobenius reciprocity together 
with (\ref{ext1})--(\ref{ext2}) or by arguing inductively via Theorem~\ref{Ses}.
They are reminiscent of the homological properties of a quasi-hereditary algebra. 

\begin{Theorem}[Kato, McNamara, Brundan-Kleshchev-McNamara]\label{hp}
Suppose we are given $\alpha$ of height $n$ and 
$\lambda,\mu \in \KP(\alpha)$.
\begin{itemize}
\item[(1)] The standard module $\Delta(\lambda)$ has an exhaustive 
descending filtration with $\bar\Delta(\lambda)$ at the top and all other sections of the form $q^{2d} \bar\Delta(\lambda)$ for $d > 0$.
\item[(2)] We have that $\operatorname{Ext}^i_{H_\alpha}(\Delta(\lambda), \Delta(\lambda)) =\operatorname{Ext}^i_{H_\alpha}(\Delta(\lambda), \bar\Delta(\lambda)) = 0$ for $i \geq 1$. 
\item[(3)] Setting $\bar\nabla(\mu) := \bar\Delta(\mu)^\circledast$, we have that
$\Dim \operatorname{Ext}^i_{H_\alpha}(\Delta(\lambda), \bar\nabla(\mu)) = \delta_{i,0} \delta_{\lambda,\mu}$ for $i \geq 0$.
\item[(4)] 
The projective dimension of 
$\Delta(\lambda)$ is bounded above by $n - l$, where $l$ is the number of parts of $\lambda$.
\item[(5)]
The global dimension of $H_\alpha$ is equal to $n$.
\end{itemize}
\end{Theorem}

\begin{Corollary}
For $\lambda \in \KP$, let $P(\lambda)$ be the projective cover of $L(\lambda)$.
Then $P(\lambda)$ has a finite filtration with sections of the form 
$\Delta(\mu)$ (up to a degree shift) and graded multiplicities satsfying
$[P(\lambda):\Delta(\mu)] = [\bar\Delta(\mu):L(\lambda)]$.
\end{Corollary}

\begin{Corollary}\label{c}
For any $\lambda \in \KP$, we have that
\begin{align*}
\Delta(\lambda) &\cong
P(\lambda) \bigg/\!\!\! 
\sum_{\phantom{q^n}\mu\not\preceq\lambda\phantom{q^n}}
\!\!\!\!\!\sum_{f:P(\mu)
  \rightarrow P(\lambda)}\!\! \operatorname{im} f,&
\quad \bar\Delta(\lambda) &
\cong P(\lambda) \bigg/\!\!\!
\sum_{\phantom{q^n}\mu\not\prec\lambda\phantom{q^n}}
\!\!\!\!\!\sum_{f:P(\mu)
  \rightarrow \operatorname{rad} P(\lambda)} \!\!\operatorname{im} f,
\end{align*}
summing over all (not necessarily homogeneous) homomorphisms $f$.
\end{Corollary}

\subsection*{Notes}
The key facts about the algebra $\f$ summarized here are all proved in Lusztig's book \cite{Lubook}.
The ``bilexicographic'' partial order $\preceq$ was introduced originally in \cite[$\S$3]{Mac}.
The triangularity (\ref{tr2}) of the bar involution on the dual PBW basis 
follows from Theorem~\ref{mac2}; more direct proofs avoiding quiver Hecke algebras also exist.

Cuspidal modules and the definition (\ref{psm}) 
arose first in the work of Kleshchev and Ram \cite{KR2}, which only treats certain rather special convex orders as in \cite{Lec}.
Kato then gave a completely different approach to the construction of cuspidal modules
working in the framework of \cite{VV}, and based on the geometric definition of 
so-called {\em Saito reflection functors} which 
categorify Lusztig's braid group action. McNamara subsequently found an elegant algebraic treatment
which can be found in \cite[Theorem 3.1]{Mac}.
The positivity of the coefficients of $p_{\lambda,\mu}(q)$ from Corollary~\ref{ka} 
was conjectured a long time ago by Lusztig and proved for the first
time in full generality by Kato in \cite{Kato}.
It is one of the first genuine applications of quiver Hecke
algebras to the theory of canonical bases.

A geometric approach to the definition of
standard modules can be found in \cite{Kato}. The construction
explained here is based instead on our article \cite{BKM}, which also
covers the non-simply-laced finite types BCFG;
non-simply-laced types are more complicated 
as we cannot exploit the naive approach of inducing from $H_\alpha'$.
The homological properties in Theorem~\ref{hp}
are proved in \cite{Kato}, \cite{Mac} and \cite{BKM}. In 
particular the finiteness of global dimension for simply-laced types
in characteristic zero was established originally by Kato; in type A
it can also be deduced from \cite{OS}.
The exact value for the global dimension of $H_\alpha$ was determined
later by McNamara; in type A this 
follows also from Theorem~\ref{unpub}.


\small

\frenchspacing
\begin{thebibliography}{22}
\bibitem{Ar}
S. Ariki, On the decomposition numbers of the Hecke algebra of $G(n,
1, m)$, 
{\em J. Math. Kyoto Univ.} {\bf 36} (1996), 789--808.

\bibitem{BGS}
A. Beilinson, V. Ginzburg and W. Soergel,
Koszul duality patterns in representation theory,
{\em J. Amer. Math. Soc.} {\bf 9} (1996), 473--527.

\bibitem{BK}
A. Berenstein and D. Kazhdan,
Geometric and unipotent crystals II: from unipotent bicrystals to
crystal bases, 
{\em Contemp. Math.} {\bf 433} (2007), 13--88.

\bibitem{BK1}
J. Brundan and A. Kleshchev,
Blocks of cyclotomic Hecke algebras and Khovanov-Lauda algebras,
{\em Invent. Math.} {\bf 178} (2009), 451--484.

\bibitem{BK2}
J. Brundan and A. Kleshchev,
Graded decomposition numbers for cyclotomic Hecke algebras,
{\em Advances Math.} {\bf 222} (2009), 1883--1942.

\bibitem{BKM}
J. Brundan, A. Kleshchev and P. McNamara,
Homological properties of finite type Khovanov-Lauda-Rouquier
algebras;
{\tt arXiv:1210.6900}.

\bibitem{CG}
N. Chriss and V. Ginzburg, {\em Representation Theory and Complex Geometry}, Birkh\"auser, 1997.

\bibitem{CR}
J. Chuang and R. Rouquier,
Derived equivalences for symmetric groups and
$\mathfrak{sl}_2$-categorification,
{\em  Ann. Math.} {\bf 167} (2008), 245--298. 

\bibitem{CF}
L. Crane and I. Frenkel, 
Four-dimensional topological quantum field theory, Hopf categories, and
the canonical bases, {\em J. Math. Phys.} {\bf 35} (1994), 5136--5154.

\bibitem{EKL}
A. Ellis, M. Khovanov and A. Lauda, 
The odd nilHecke algebra and its diagrammatics;
{\tt arXiv:1111.1320}.

\bibitem{GV}
I. Grojnowski and M. Vazirani, Strong multiplicity one theorems for 
affine Hecke algebras of type A, {\em Transform. Groups} {\bf 6}, 143--155, 2001.


\bibitem{HL}
D. Hernandez and B. Leclerc, 
Quantum Grothendieck groups and derived Hall algebras;
{\tt arXiv:1109.0862}.

\bibitem{KK}
S.-J. Kang and M. Kashiwara, Categorification of highest weight
modules via Khovanov-Lauda-Rouquier algebras, {\em Invent. Math.} {\bf 190} (2012), 699--742. 

\bibitem{KKK}
S.-J. Kang, M. Kashiwara and M. Kim,
$R$-matrices for quantum affine algebras and Khovanov-Lauda-Rouquier
algebras I;
{\tt arXiv:1209.3536}.

\bibitem{KKO}
S.-J. Kang, M. Kashiwara and S.-J.  Oh,
Supercategorification of quantum Kac-Moody  algebras,
{\em Advances Math.} {\bf 242} (2013), 116--162.

\bibitem{KKT}
S.-J. Kang, M. Kashiwara and S. Tsuchioka, Quiver Hecke superalgebras;
{\tt arXiv:1107.1039}.

\bibitem{Kato}
S. Kato,
PBW bases and KLR algebras;
{\tt arXiv:1203.5254}.

\bibitem{KL1}
M. Khovanov and A. Lauda,
A diagrammatic approach to categorification of quantum
groups I, {\em Represent. Theory} {\bf 13} (2009), 309--347. 

\bibitem{KL2}
M. Khovanov and A. Lauda,
A diagrammatic approach to categorification of quantum
groups II, {\em Trans. Amer. Math. Soc.} {\bf 363} (2011), 2685--2700. 

\bibitem{KL3}
M. Khovanov and A. Lauda,
A diagrammatic approach to categorification of quantum groups III,
{\em Quantum Top.} {\bf 1} (2010), 1--92.

\bibitem{Kbook}
 A. Kleshchev, {\em Linear and Projective Representations of Symmetric 
Groups}, CUP, 2005. 

\bibitem{KR1}
A. Kleshchev and A. Ram, 
Homogeneous representations of Khovanov-Lauda algebras, {\em
  J. Eur. Math. Soc.}
{\bf 12} (2010), 1293--1306.

\bibitem{KR2}
A. Kleshchev and A. Ram, 
Representations of Khovanov-Lauda-Rouquier algebras and combinatorics
of Lyndon words, {\em Math. Ann.} {\bf 349} (2011), 943--975. 

\bibitem{Lauda}
A. Lauda, A categorification of quantum $\mathfrak{sl}(2)$, {\em
  Advances Math.} {\bf 225} (2010), 3327--3424.


\bibitem{Lec}
B. Leclerc, Dual canonical bases, quantum shuffles and $q$-characters,
{\em Math. Z.} {\bf 246} (2004), 691--732. 


\bibitem{Lubook}
G. Lusztig, {\em Introduction to Quantum Groups}, Birkh\"auser, 1993.


\bibitem{Mac}
P. McNamara,
Finite dimensional representations of Khovanov-Lauda-Rouquier algebras
I: finite type, to appear in {\em J. Reine Angew. Math.};
{\tt arxiv:1207.5860}.

\bibitem{NV}
C. N\v ast\v asescu and F. Van Oystaeyen,
{\em Methods of Graded Rings}, Lecture Notes in Math. 1836,
Springer, 2004.

\bibitem{OS}
E. Opdam and M. Solleveld,  
Homological algebra for affine Hecke algebras, {\em Advances Math.} {\bf 220} (2009), 1549--1601.

\bibitem{R1}
R. Rouquier, 
2-Kac-Moody algebras;
{\tt arXiv:0812.5023}.

\bibitem{R2}
R. Rouquier,
Quiver Hecke algebras and $2$-Lie algebras,
{\em Algebra Colloq.} {\bf 19} (2012), 359--410.

\bibitem{VV}
M. Varagnolo and E. Vasserot,  
Canonical bases and KLR-algebras, {\em J. Reine
Angew. Math.} {\bf 659} (2011), 67--100.

\bibitem{VV2}
M.
 Varagnolo and E. Vasserot, Canonical bases and affine Hecke algebras
 of type B,
{\em Invent. Math.} {\bf 185} (2011), 593--693. 

\bibitem{Wang}
W. Wang, Spin Hecke algebras of finite and affine types, 
{\em Advances Math.} {\bf 212} (2007), 723--748.

\bibitem{W1}
B. Webster,
Knot invariants and higher representation theory;
{\tt arXiv:1309.3796}.

\bibitem{W2}
B. Webster, 
Canonical bases and higher representation theory; {\tt arXiv:1209.0051}.

\bibitem{W}
G. Williamson,
On an analogue of the James conjecture;
{\tt arXiv:1212.0794}.


\bibitem{Zel1}
A. Zelevinsky, Induced representations of reductive $p$-adic groups
II, 
{\em Ann. Sci. E.N.S.} {\bf 13} (1980), 165--210.

\bibitem{Zel2}
A. Zelevinsky, A $p$-adic analog of the Kazhdan-Lusztig conjecture,
{\em Funct. Anal. Appl.} {\bf 15} (1981), 83--92.

\end{thebibliography}
\end{document}